\documentclass[11pt]{article}
\usepackage{amsfonts}
\usepackage{latexsym,amssymb,amsmath,graphics,cite,arydshln}
\usepackage{tikz}
\usepackage{fancyhdr}
\usepackage{color}

\usepackage{multirow} 
\usepackage{xcolor}
\usepackage{amssymb}
\usepackage{geometry}
\usepackage{lipsum}
\usepackage{enumerate}

\newcommand{\ignore}[1]{}

\begin{document}
\newcommand{\qed}{\hphantom{.}\hfill $\Box$\medbreak}
\newcommand{\proof}{\noindent{\bf Proof \ }}

\newtheorem{theorem}{Theorem}[section]
\renewcommand{\theequation}{\thesection.\arabic{equation}}
\newtheorem{lemma}[theorem]{Lemma}
\newtheorem{corollary}[theorem]{Corollary}
\newtheorem{example}[theorem]{Example}
\newtheorem{remark}[theorem]{Remark}
\newtheorem{definition}[theorem]{Definition}
\newtheorem{construction}[theorem]{Construction}
\newtheorem{fact}[theorem]{Fact}
\newtheorem{proposition}[theorem]{Proposition}
\newtheorem{conjecture}[theorem]{Conjecture}

\title{ The maximum product of sizes of cross-\(t\)-intersecting families 
}

\author{\small Jingjun\ Bao and Lijun\ Ji
\thanks{J. Bao is with School of Mathematics and Statistics, Ningbo University, Ningbo 315211, P. R. China;
 L. Ji is with the Department of Mathematics, Soochow University,
 Suzhou 215006, P. R. China (email: baojingjun@hotmail.com; jilijun@suda.edu.cn).}
}

%
\date{}
\maketitle

\begin{abstract}
Two families of sets \(\mathcal{A}\) and \(\mathcal{B}\) are called \emph{cross-\(t\)-intersecting} if \(|A \cap B| \geq t\) for all \(A \in \mathcal{A}\) and \(B \in \mathcal{B}\). Determining the maximum product of sizes for such cross-\(t\)-intersecting families is an active problem in extremal set theory. In this paper,  we verify the following cross-\(t\)-intersecting version of the Erd\H{o}s–Ko–Rado theorem: For \(k\geq l \geq t \geq 3\) and \(\min\{m,n\} \geq (t+1)(k-t+1)\), the maximun value of \(|\mathcal{A}||\mathcal{B}|\) for two  cross-\(t\)-intersecting families \(\mathcal{A}\subseteq \binom{[n]}{k}\) and \(\mathcal{B} \subseteq \binom{[m]}{l}\) is \( \binom{n-t}{k-t}\binom{m-t}{l-t}\). Moreover, we characterize the extremal families attaining the upper bound. Our result confirms a conjecture of Tokushige for \(t \geq 3\), and actually proves a more general result.

\smallskip
\end{abstract}

\section{Introduction}
Let \(n\), \(k\) and \(t\) be positive integers with \(t \leq k \leq n\). We denote the set \(\{1, \ldots, n\}\) by \([n]\) and use \(2^{[n]}\) to represent its power set. A nonempty subset \(\mathcal{F}\) of \(2^{[n]}\) is called \(k\)-{\em uniform} if its elements all have size \(k\). For any finite set \(D\), we use \(\binom{D}{k}\) to denote the collection of all \(k\)-element subsets of \(D\). We say that \(\mathcal{F}\) is a \(t\)-{\em intersecting family} if \(|A\cap B|\geq t\) for all \(A,B\in\mathcal{F}\). A \(1\)-intersecting family is simply referred to as an intersecting family. The Erd\H{o}s–Ko–Rado Theorem is one of the central results in extremal combinatorics.

\begin{theorem}[Erd\H{o}s–Ko–Rado Theorem \cite{EKR1961}]\label{EKR}
Let \(k\), \(n\) and \(t\) be positive integers such that \(t \leq k \leq n\). If \(\mathcal{F}\subseteq\binom{[n]}{k}\) is \(t\)-intersecting, then there is a constant $n_0(k,t)$ such that for \(n\geq n_{0}(k,t)\), the following inequality holds. \[|\mathcal{F}|\leq\binom{n - t}{k - t}.\]
\end{theorem}

Let \(k, t\) be fixed positive integers with \(t \leq k\). Denote by \(N_{0}(k,t)\) the smallest possible value of \(n_{0}(k,t)\) in Theorem \ref{EKR}. For the case \(t = 1,\) it was shown in \cite{EKR1961} that \(N_{0}(k,1)=2k\). For \(t\geq 1,\) the value of \(N_{0}(k,t)\) is given by \(N_{0}(k,t)=(t + 1)(k - t+1)\), which was established in \cite{F1978} for \(t\geq15\) and in \cite{W1984} for all \(t\).

Let \(r\) be an integer such that \(0\leq r\leq\frac{n - t}{2}\). Define the following \(k\)-uniform family 
\[\mathcal{F}(n,k,t,r):=\left\{A\in\binom{[n]}{k}:|A\cap[t + 2r]|\geq t + r\right\}.\] 
It is a \(t\)-intersecting family, since for its elements \(A\), \(B\) we have \(|A\cap B|\geq|A\cap[t + 2r]|+|B\cap[t + 2r]|-(t + 2r)\geq t\). Frankl \cite{F1978} conjectured that if \(\mathcal{F}\) is a \(t\)-intersecting subfamily of \(\binom{n}{k}\), then \[|\mathcal{F}|\leq\max_{0\leq r\leq(n - t)/2}|\mathcal{F}(n,k,t,r)|.\]

This conjecture was proved partially by Frankl and Füredi \cite{FF1986}, and then settled completely by Ahlswede and Khachatrian \cite{AK1997}.

\begin{theorem}[Ahlswede and Khachatrian \cite{AK1997}]\label{AK97}
Let \(n, k, t\) be positive integers such that \(t \leq k \leq n\). Let \(\mathcal{F}\) be a \(t\)-intersecting \(k\)-uniform family of subsets of \([n]\).
\begin{itemize}
\item[(i)] If \((k - t + 1)(2+\frac{t - 1}{r+1})<n<(k - t + 1)(2+\frac{t - 1}{r})\) for some nonnegative integer \(r\), then \(|\mathcal{F}|\leq|\mathcal{F}(n,k,t,r)|\) and the equality holds if and only if \(\mathcal{F}=\mathcal{F}(n,k,t,r)\) up to isomorphism.
\item[(ii)] If \((k - t + 1)(2+\frac{t - 1}{r+1})=n\) for some nonnegative integer \(r\), then \(|\mathcal{F}|\leq|\mathcal{F}(n,k,t,r)| = |\mathcal{F}(n,k,t,r + 1)|\) and the equality holds if and only if \(\mathcal{F}=\mathcal{F}(n,k,t,r)\) or \(\mathcal{F}(n,k,t,r + 1)\) up to isomorphism.
\end{itemize}
\end{theorem}

Let \(\mathcal{A}\) and \(\mathcal{B}\) be two families of subsets of \([n]\). They are called {\em cross-\(t\)-intersecting} if \(|A\cap B|\geq t\) for all \(A\in\mathcal{A}\) and \(B\in\mathcal{B}\). In the special case where \(t = 1,\) we say that \(\mathcal{A}\) and \(\mathcal{B}\) are {\em cross-intersecting}. The cross-\(t\)-intersecting property can be seen as a natural generalization of the \(t\)-intersecting property, as the two properties coincide when \(\mathcal{A}=\mathcal{B}\). 

Now, consider cross-\(t\)-intersecting families \(\mathcal{A}\subseteq \binom{[n]}{k}\) and \(\mathcal{B}\subseteq \binom{[n]}{l}\). We say that these families are {\em maximum} if there do not exist cross-\(t\)-intersecting families \(\mathcal{A}_{1}\subseteq \binom{[n]}{k}\) and \(\mathcal{B}_{1}\subseteq \binom{[n]}{l}\) such that  \(|\mathcal{A}||\mathcal{B}|<|\mathcal{A}_{1}||\mathcal{B}_{1}|\). Similarly, we say that \(\mathcal{A}\) and \(\mathcal{B}\) are {\em maximal} if there do not exist cross-\(t\)-intersecting families \(\mathcal{A}\subseteq \mathcal{A}_{1} \subset {[n]\choose k}$ and $\mathcal{B}\subseteq \mathcal{B}_{1} \subset {[n]\choose l}\) such that  \(|\mathcal{A}||\mathcal{B}|<|\mathcal{A}_{1}||\mathcal{B}_{1}|\). 

The study of possible maximum sizes of pairs of cross-intersecting families is one of the central problems of extremal set theory. In 1986, Pyber \cite{P1986} employed the method of cyclic permutations to extend the Erd\H{o}s–Ko–Rado Theorem to the cross-intersecting case.

\begin{theorem}[Pyber \cite{P1986}]\label{P86}
Let \(n, k, l\) be positive integers such that \(l \leq  k \leq n\). Suppose that \(\mathcal{A}\subseteq\binom{[n]}{k}\) and \(\mathcal{B}\subseteq\binom{[n]}{l}\) are two cross-intersecting families.
\begin{itemize}
\item[{\rm (1)}] If \(n\geq2k+l-2\), then \(|\mathcal{A}||\mathcal{B}|\leq\binom{n - 1}{k - 1}\binom{n - 1}{l - 1}\).
\item[{\rm (2)}] If \(k=l\) and \(n\geq2k\), then \(|\mathcal{A}||\mathcal{B}|\leq\binom{n - 1}{k - 1}^2\).
\end{itemize}
\end{theorem}

For \(k> l,\) the lower bound on \(n\) in (1) of Theorem \ref{P86} turns out not to be sharp. In 1989, Matsumoto and Tokushige \cite{MT1989} derived a sharper result.

\begin{theorem}[Matsumoto and Tokushige \cite{MT1989}]
Let \(n, k, l\) be positive integers such that \(n\geq 2k \geq 2l\). If \(\mathcal{A}\subseteq\binom{[n]}{k}\) and \(\mathcal{B}\subseteq\binom{[n]}{l}\) are cross-intersecting, then \(|\mathcal{A}||\mathcal{B}|\leq\binom{n - 1}{k - 1}\binom{n - 1}{l - 1}\). Moreover, the equality holds if and only if \(\mathcal{A}=\left\{A\in\binom{[n]}{k}:x\in A\right\}\) and \(\mathcal{B}=\left\{B\in\binom{[n]}{l}:x\in B\right\}\) for some fixed element \(x\in[n]\).
\end{theorem}

Gromov \cite{G2010} found an application of these inequalities to geometry. For the general cross-\(t\)-intersecting case, Tokushige \cite{T2010} employed a combinatorial approach to derive the following analogous result.

\begin{theorem}[Tokushige \cite{T2010}]\label{T10}
Let \(n, k, t\) be positive integers such that \(t \leq k \leq n\) and \(n>2tk\). If \(\mathcal{A}\subseteq\binom{[n]}{k}\) and \(\mathcal{B}\subseteq\binom{[n]}{k}\) are cross-\(t\)-intersecting, then \(|\mathcal{A}||\mathcal{B}|\leq\binom{n - t}{k - t}^2\). Moreover, the equality holds if and only if \(\mathcal{A}=\mathcal{B}=\left\{F\in\binom{[n]}{k}:T\subset F\right\}\) for some \(t\)-subset \(T\) of \([n]\).
\end{theorem}

In 2013, Tokushige \cite{T2013} further refined this research by using the eigenvalue method to improve the lower bound on \(n\) for cross-\(t\)-intersecting families. 

\begin{theorem}[Tokushige \cite{T2013}]\label{T13}
Let \(n, k, t\) be positive integers such that \(t \leq k \leq n\) and \(\frac{k}{n}<1-\frac{1}{\sqrt[t]{2}}\). If \(\mathcal{A}\subseteq\binom{[n]}{k}\) and \(\mathcal{B}\subseteq\binom{[n]}{k}\) are cross-\(t\)-intersecting, then \(|\mathcal{A}||\mathcal{B}|\leq\binom{n - t}{k - t}^2\). Moreover, the equality holds if and only if \(\mathcal{A}=\mathcal{B}=\left\{F\in\binom{[n]}{k}:T\subset F\right\}\) for some \(t\)-subset \(T\) of \([n]\).
\end{theorem}

In the same paper \cite{T2013}, Tokushige conjectured that the lower bound on \(n\) in Theorem \ref{T13} could be improved to \(n\geq(t + 1)(k - t + 1)\). In 2014, Frankl, Lee, Siggers and Tokushige \cite{FLST2014} verified this  conjecture under the stronger assumptions that \(t\geq14\) and \(n\geq(t + 1)k\). In the same year, Borg \cite{B2014} independently confirmed the conjecture for sufficiently large \(n\). More recently, Zhang and Wu \cite{ZW2025} used shifting techniques and generating set analysis to confirm the conjecture for $t\geq 3.$ Also in 2025, Tanaka and Tokushige \cite{TT2025} employed a semidefinite programming approach to resolve the case $t=2.$

In 2013, Tokushige \cite{T2013} proposed a conjecture on a lower bound on \(n\) for non-uniform cross-\(t\)-intersecting families.
\begin{conjecture}[Tokushige \cite{T2013}]\label{TC13}
Let \(k\geq l\geq t\geq 1\) and \(n\geq (t+1)(k-t+1)\). Suppose that two families \(\mathcal{A}\subseteq\binom{[n]}{k}\) and \(\mathcal{B}\subseteq\binom{[n]}{l}\) are cross-\(t\)-intersecting. Then \(|\mathcal{A}||\mathcal{B}|\leq\binom{n - t}{k - t}\binom{n - t}{l - t}\).
\end{conjecture}

In 2016, Borg \cite{B2016} verified the conjecture under the stronger assumptions that \(n\geq (t + 4)(k-t)+l-1\) and established the following general result.

\begin{theorem}[Borg \cite{B2016}]
Let \(m, n, k, l\) be positive integers such that \(1\leq t\leq l\leq k\),\ \(l\leq m\),\ \(k\leq n\) and \(\min\{m,n\}\geq(t + 4)(k-t)+l-1\). If \(\mathcal{A}\subseteq\binom{[n]}{k}\) and \(\mathcal{B}\subseteq\binom{[m]}{l}\) are cross-\(t\)-intersecting, then \(|\mathcal{A}||\mathcal{B}|\leq\binom{m - t}{l- t}\binom{n - t}{k - t}\).
\end{theorem}

Note that in the precise expression of Borg's result, the condition on \(\min\{m,n\}\) is improved from \((t + 4)(k- t)+ l- 1\) to \((t + 2)(k-t)+ l- 1\) as \(t\) increases from \(1\) to \(7\). In 2025, He et al. \cite{HLWZ2026} employed shifting techniques and generating set analysis to confirm the conjecture for $t\geq 3.$ More recently, Chen et al. \cite{CLWZ2025} also used the generating set method to verify the Conjecture \ref{TC13} for $t=2$ and $n\geq 3.38l$. A natural generalization of Conjecture \ref{TC13} replaces the condition \(\mathcal{A}\subseteq\binom{[n]}{k}, \mathcal{B}\subseteq\binom{[n]}{l}\) with \(\mathcal{A}\subseteq\binom{[n]}{k}, \mathcal{B}\subseteq\binom{[m]}{l}\).  This general problem remains wide open. In this paper, we resolve it completely for all \(t\geq 3\) and \(\min\{m,n\}\geq (t+1)(k-t+1)\). Our main result is stated below.

\begin{theorem}\label{TM}
Let \(n, m, l, k\) be positive integers such that \(3\leq t\leq l\leq k\) and \(\min\{m, n\} \geq(t+1)(k-t+1)\). If \(\mathcal{A}\subseteq\binom{[n]}{k}\) and \(\mathcal{B}\subseteq\binom{[m]}{l}\) are two cross-\(t\)-intersecting families, then \(|\mathcal{A}||\mathcal{B}|\leq\binom{n-t}{k-t}\binom{m-t}{l-t}\). Moreover, the equality holds if and only if one of the following holds: 
\begin{itemize}
\item[{\rm (1)}]\(\mathcal{A}=\left\{A \in \binom{[n]}{k} : T \subset A \right\}\) and \(\mathcal{B}=\left\{B \in \binom{[m]}{l} : T \subset B \right\}\) for some \(t\)-element subset \(T\);
\item[{\rm (2)}] \(n=m=(t+1)(k-t+1),\ k=l\) and \(\mathcal{A}=\mathcal{B}=\left\{A \in \binom{[n]}{k} :  |A \cap T| \geq t+1\right\}\) for some \((t+2)\)-element subset \(T\).
\end{itemize}
\end{theorem}

The methods used in this paper are the shift operator method \cite{EKR1961} and the generating set method \cite{AK1996, AK1997}. These techniques are versatile and can also be applied to address other types of extremal set problems. For instance, they can be used to determine the maximum product of sizes of cross-intersecting multi-part families, as well as the maximum sum of sizes of cross-intersecting families. The definitions and fundamental properties of these methods will be introduced in Section 2. In Section 3, we will present three crucial inequalities. In Section 4, we will provide a complete proof of our main result.  

\section{Preliminaries}
Let $\mathcal{A}$ be a family consisting of $k$-element subsets of $[n]$. For $i,j\in[n]$ and $A\in\mathcal{A}$, define 
\[
S_{ij}(A)= 
\begin{cases}
(A\backslash\{j\})\cup\{i\}, & \text{if } j\in A, i\notin A,(A\backslash\{j\})\cup\{i\}\notin \mathcal{A}; \\
A, & \text{otherwise,}
\end{cases}
\]
and set $S_{ij}(\mathcal{A}) = \{S_{ij}(A):A\in\mathcal{A}\}$ correspondingly. The procedure to obtain $S_{ij}(\mathcal{A})$ from $\mathcal{A}$ is called the {\em shift operation}, which was first introduced in \cite{EKR1961} (see also \cite{F1987}). We observe that $S_{ij}(\mathcal{A})$ has the same cardinality as $\mathcal{A}$ and is also $k$-uniform. We say that $\mathcal{A}$ is {\em left-compressed} if $S_{ij}(\mathcal{A})=\mathcal{A}$ for all ordered pairs $(i,j)$ with $1\leq i < j\leq n$. The following fact is well-known. 

\begin{fact}[\cite{F1987}] \label{lcl}
Let $\mathcal{A}$ and $\mathcal{B}$ be two families of subsets of $[n]$. If $\mathcal{A}$ and $\mathcal{B}$ are cross $t$-intersecting, then $S_{i,j}(\mathcal{A})$ and $S_{i,j}(\mathcal{B})$ are cross-$t$-intersecting with $|S_{i,j}(\mathcal{A})| = |\mathcal{A}|$ and $|S_{i,j}(\mathcal{B})| = |\mathcal{B}|$.
\end{fact}

For a subset $E$ of $[n]$, we set $s^{+}(E)=\max\{i:i\in E\}$, $\mathcal{U}(E)=\{A\subseteq[n]:E\subseteq A\}$ and $\mathcal{D}(E)=\left\{B\in\binom{[n]}{k}:B\cap[s^{+}(E)] = E\right\}$. For $\mathcal{E}\subseteq 2^{[n]},$ we set $s^{+}(\mathcal{E})=\max\{s^{+}(E):E\in\mathcal{E}\}$, $\mathcal{U}(\mathcal{E})=\bigcup_{E\in\mathcal{E}}\mathcal{U}(E)$ and $\mathcal{D}(\mathcal{E})=\bigcup_{E\in\mathcal{E}}\mathcal{D}(E)$ correspondingly. Let $\emptyset\neq \mathcal{F}\subset \binom{[n]}{k}$, we say that a family $g(\mathcal{F})\subseteq\bigcup_{i\leq k}\binom{[n]}{i}$ is a {\em generating set} of $\mathcal{F}$ if $\mathcal{U}(g(\mathcal{F}))\cap\binom{[n]}{k}=\mathcal{F}$. It is clear that $\mathcal{F}$ is a generating set of itself. The set of all generating sets of $\mathcal{F}$ forms a nonempty set, which we denote by $G(\mathcal{F})$. The notion of generating sets of a $k$-uniform family was first introduced in \cite{AK1997}. For $g(\mathcal{F})\in G(\mathcal{F})$, let $g_{*}(\mathcal{F})$ be the set of all minimal elements (in the sense of set-theoretical inclusion) of $g(\mathcal{F})$. Set $G_{*}(\mathcal{F})=\{g(\mathcal{F})\in G(\mathcal{F}):g(\mathcal{F}) = g_{*}(\mathcal{F})\}$. Clearly, $G_{*}(\mathcal{F})\neq \emptyset$. For $g(\mathcal{F})\in G_{*}(\mathcal{F})$, let $s = s^{+}(g(\mathcal{F}))$, set $g^{*}(\mathcal{F})=\{E\in g(\mathcal{F}):s\in E\}$, $g_{i}^{*}(\mathcal{F})=\{E\in g^{*}(\mathcal{F}):|E| = i\}$ and $g_{i}^{*}(\mathcal{F})'=\{E\backslash\{s\}:E\in g_{i}^{*}(\mathcal{F})\}$ for $1\leq i\leq s$. From \cite{AK1997}, we know that the generating sets have the following properties.

\begin{lemma}[Ahlswede and Khachatrian \cite{AK1997}]\label{AKGS}
Let \(\mathcal{F}\) be a left-compressed \(t\)-intersecting subfamily of \(\binom{[n]}{k}\), \(g(\mathcal{F}) \in G_{*}(\mathcal{F})\) and \(s = s^{+}(g(\mathcal{F}))\). Then the following statements hold.

\begin{enumerate}[(i)]
    \item If \(n > 2k - t\), then \(|E_1 \cap E_2| \geq t\) for all \(E_1, E_2 \in g(\mathcal{F})\).

    \item For \(1 \leq i < j \leq s\) and \(E \in g(\mathcal{F})\), one has either \(S_{ij}(E) \in g(\mathcal{F})\) or \(F \subset S_{ij}(E)\) for some \(F \in g(\mathcal{F})\).

    \item \(\mathcal{F}\) is a disjoint union \(\mathcal{F} = \bigcup_{E \in g(\mathcal{F})} \mathcal{D}(E)\).

    \item If \(\mathcal{F}\) is maximal, then for any \(E_1, E_2 \in g^{*}(\mathcal{F})\) with \(|E_1 \cap E_2| = t\), necessarily \(|E_1| + |E_2| = s + t\) and \(E_1 \cup E_2 = [s]\). Furthermore, if \(g^{*}_{i}(\mathcal{F}) \neq \emptyset\), then \(g^{*}_{s + t - i}(\mathcal{F}) \neq \emptyset\) and for any \(E_1 \in g^{*}_{i}(\mathcal{F})\), there exists \(E_2 \in g^{*}_{s + t - i}(\mathcal{F})\) with \(|E_1 \cap E_2| = t\) and \(E_1 \cup E_2 = [s]\).

    \item If \(g^{*}_{i}(\mathcal{F}) \neq \emptyset\), then \(\mathcal{F}_1 = \mathcal{F} \cup \mathcal{D}(g^{*}_{i}(\mathcal{F})) \setminus \mathcal{D}(g^{*}_{s + t - i}(\mathcal{F}))\) is also a \(t\)-intersecting subfamily of \(\binom{[n]}{k}\) with
    \[
    |\mathcal{F}_1| = |\mathcal{F}| + |g^{*}_{i}(\mathcal{F})| \binom{n - s}{k - i + 1} - |g^{*}_{s + t - i}(\mathcal{F})| \binom{n -s}{k + i -s - t}.
    \]
\end{enumerate}
\end{lemma}

Let $\mathcal{A}$ and $\mathcal{B}$ be two cross-$t$-intersecting subfamilies of $\binom{[n]}{k}$. For $n>2k-t$, it is easy to see that $\mathcal{A}$ and $\mathcal{B}$ are cross-$t$-intersecting if and only if $g(\mathcal{A})$ and $g(\mathcal{B})$ are cross-$t$-intersecting for all $g(\mathcal{A})\in G(\mathcal{A})$ and $g(\mathcal{B})\in G(\mathcal{B})$. Zhang and Wu \cite{ZW2025} obtained the following properties of the generating sets of $\mathcal{A}$ and $\mathcal{B}$.

\begin{lemma}[Zhang and Wu \cite{ZW2025}]\label{ZW}
Let $\mathcal{A}$ and $\mathcal{B}$ be two maximal left-compressed cross-$t$-intersecting subfamilies of $\binom{[n]}{k}$ with $n>2k-1$. Let $g(\mathcal{A})\in G_{*}(\mathcal{A})$ and $g(\mathcal{B})\in G_{*}(\mathcal{B})$ such that $s:=\max\{s^{+}(g(\mathcal{A})),s^{+}(g(\mathcal{B}))\}$ is minimal. Then the following statements hold.
\begin{enumerate}[(i)]
    \item For $1\leq i<j\leq s$, $\mathcal{F}\in\{\mathcal{A},\mathcal{B}\}$ and $E\in g(\mathcal{F})$, one has either $S_{ij}(E)\in g(\mathcal{F})$ or $F\subset S_{ij}(E)$ for some $F\in g(\mathcal{F})$.
    \item For $t\leq i\leq k$, $g_{i}^{*}(\mathcal{A})\neq\emptyset$ if and only if $g_{s + t - i}^{*}(\mathcal{B})\neq\emptyset$. Furthermore, for each $E\in g_{i}^{*}(\mathcal{A})$, there exists $F\in g_{s + t - i}^{*}(\mathcal{B})$ such that $|E\cap F| = t$ and $E\cup F=[s]$.
    \item If $g_{i}^{*}(\mathcal{A})\neq\emptyset$, then the families $\mathcal{A}_{1}=\mathcal{A}\cup \mathcal{D}(g_{i}^{*}(\mathcal{A})')\subset \binom{[n]}{k}$ and $\mathcal{B}_{1}=\mathcal{B}\setminus \mathcal{D}(g_{s + t - i}^{*}(\mathcal{B}))\subset \binom{[n]}{k}$ remain cross-$t$-intersecting families with sizes given by
    \[|\mathcal{A}_{1}|=|\mathcal{A}|+|g_{i}^{*}(\mathcal{A})|\binom{n - s}{k - i+1}\quad\text{and}\quad|\mathcal{B}_{1}|=|\mathcal{B}|-|g_{s + t - i}^{*}(\mathcal{B})|\binom{n - s}{k + i - s - t}.\]
\end{enumerate}
\end{lemma}

It is easy to generalize the cases for the non-uniform situation. Let \(\mathcal{A}\subseteq \binom{[n]}{k}\) and \(\mathcal{B}\subseteq \binom{[m]}{l}\) be two cross-$t$-intersecting families. For $\min\{m, n\}>k+l-t$, it is easy to find that $\mathcal{A}$ and $\mathcal{B}$ are cross-$t$-intersecting if and only if $g(\mathcal{A})$ and $g(\mathcal{B})$ are cross-$t$-intersecting for all $g(\mathcal{A})\in G(\mathcal{A})$ and $g(\mathcal{B})\in G(\mathcal{B})$. By applying the shift operations if necessary, we may assume without loss of generality that both $\mathcal{A}$ and $\mathcal{B}$ are left-compressed. Similarly, for the generating sets of $\mathcal{A}$ and $\mathcal{B}$, we have the following result.

\begin{lemma}\label{lemkl}
Let \(\mathcal{A}\subseteq \binom{[n]}{k}\) and \(\mathcal{B}\subseteq \binom{[m]}{l}\) be two maximal left-compressed cross-$t$-intersecting families with $\min\{m, n\}>k+l-t$, $g(\mathcal{A})\in G_{*}(\mathcal{A})$, $g(\mathcal{B})\in G_{*}(\mathcal{B})$ such that $s:=\max\{s^{+}(g(\mathcal{A})),s^{+}(g(\mathcal{B}))\}$ is minimal. Then the following statements hold.
\begin{enumerate}[(i)]
    \item For $1\leq i<j\leq s$, $\mathcal{F}\in\{\mathcal{A},\mathcal{B}\}$ and $E\in g(\mathcal{F})$, one has either $S_{ij}(E)\in g(\mathcal{F})$ or $F\subset S_{ij}(E)$ for some $F\in g(\mathcal{F})$.
    \item For $t\leq i\leq k$, $g_{i}^{*}(\mathcal{A})\neq\emptyset$ if and only if $g_{s + t - i}^{*}(\mathcal{B})\neq\emptyset$. Furthermore, for each $E\in g_{i}^{*}(\mathcal{A})$, there exists $F\in g_{s + t - i}^{*}(\mathcal{B})$ such that $|E\cap F| = t$ and $E\cup F=[s]$.
    \item If $g_{i}^{*}(\mathcal{A})\neq\emptyset$, then $\mathcal{A}_{1}=\mathcal{A}\cup \mathcal{D}(g_{i}^{*}(\mathcal{A})')\subseteq \binom{[n]}{k}$ and $\mathcal{B}_{1}=\mathcal{B}\setminus \mathcal{D}(g_{s + t - i}^{*}(\mathcal{B}))\subseteq \binom{[m]}{l}$ are also cross-$t$-intersecting families with 
    \[|\mathcal{A}_{1}|=|\mathcal{A}|+|g_{i}^{*}(\mathcal{A})|\binom{n - s}{k - i+1}\quad\text{and}\quad|\mathcal{B}_{1}|=|\mathcal{B}|-|g_{s + t - i}^{*}(\mathcal{B})|\binom{m-s}{l + i - s - t}.\]
\end{enumerate}
\end{lemma}

\begin{proof}
Conclusion (i) follows from the second statement of Lemma \ref{AKGS}. 

For conclusion (ii), we first show that there exist some sets $A\in g(\mathcal{A})$ and $B\in g(\mathcal{B})$ such that \(|A\cap B| = t\) with \(s\in A\cap B\). Otherwise, the collections \(\{A \setminus \{s\}:A \in g(\mathcal{A})\}\) and $\{B\setminus \{s\}: B\in g(\mathcal{B})\}$ are cross-$t$-intersecting. Combining the assumption that $\mathcal{A}$ and $\mathcal{B}$ are maximal, we have $\{A\setminus \{s\}:A\in g(\mathcal{A})\}\in G_{*}(\mathcal{A})$ and $\{B\setminus \{s\}:B\in g(\mathcal{B})\}\in G_{*}(\mathcal{B})$, contradicting the minimality of $s$. Furthermore, $A\cup B=[s]$. Otherwise, if there exists an element $x\in[s]\setminus (A\cup B)$, then by (i), there exists some $A_{1}\in g(\mathcal{A})$ such that $A_{1}\subset A\cup\{x\}\setminus \{s\}$. However, this would imply $|A_{1}\cap B|\leq t - 1$, contradicting the fact that $g(\mathcal{A})$ and $g(\mathcal{B})$ are cross-$t$-intersecting. Therefore, $s\leq k+l-t$. Suppose that $g_i^{*}(\mathcal {A})\neq \emptyset$ and let $E\in g_{i}^{*}(\mathcal{A})$ and define $E' = E\setminus \{s\}$, we observe that $\mathcal{D}(E)\subsetneq \mathcal{D}(E')$. Combining this with the assumption that $g(\mathcal{A})\in G_{*}(\mathcal{A})$, we conclude that $\mathcal{A}\subsetneq\mathcal{A}\cup \mathcal{D}(E')$. If $|E'\cap F|\geq t$ for all $F\in g(\mathcal{B})$, then $\mathcal{A}_{1}=\mathcal{A}\cup \mathcal{D}(E')$ and $\mathcal{B}$ are also cross-$t$-intersecting with $\mathcal{A}\subsetneq\mathcal{A}_{1}$, contradicting the maximality of $\mathcal{A}$ and $\mathcal{B}$. Therefore, there must exist some $F_{0}\in g(\mathcal{B})$ such that $|E'\cap F_{0}|\leq t-1$. However, since $|E\cap F_{0}|\geq t$, it follows that $|E\cap F_{0}|=t$ and $s\in E\cap F_{0}$. Similarly, we have $E\cup F_{0}=[s]$. Thus, $|F_{0}|=s+t-i$ and $g_{s+t-i}^{*}(\mathcal{B})\neq\emptyset$. Similarly, if $g_{s+t-i}^*(\mathcal{B})\neq \emptyset$, then $g_i^{*}(\mathcal {A})\neq \emptyset$. Hence, the second conclusion holds.

For conclusion (iii), we only prove the first part, as the latter part follows from (v) of Lemma \ref{AKGS}. It suffices to prove that $g_{i}^{*}(\mathcal{A})'$ and $g(\mathcal{B})\setminus g_{s + t - i}^{*}(\mathcal{B})$ are cross $t$-intersecting. Let $A\in g_{i}^{*}(\mathcal{A})'$ and $B\in g(\mathcal{B})\setminus g_{s + t - i}^{*}(\mathcal{B})$. If $s\notin B$, then $|B\cap A|=|B\cap(A\cup\{s\})|\geq t$ since $g(\mathcal{A})$ and $g(\mathcal{B})$ are cross-$t$-intersecting, and $A\cup\{s\}\in g(\mathcal{A})$. Otherwise, $s\in B$. Suppose on the contrary that $|A\cap B|\leq t - 1$. Then $|(A\cup\{s\})\cap B| = t$. A proof process similar to that of (ii) shows that $(A\cup\{s\})\cup B=[s]$. It follows that $B\in g_{s + t - i}^{*}(\mathcal{B})$, which contradicts the assumption that $B\in g(\mathcal{B})\setminus g_{s + t - i}^{*}(\mathcal{B})$. Thus, $|A\cap B|\geq t$. This completes the proof. \qed
\end{proof}

In the following, we introduce a result that will be needed in Section 4.


\begin{theorem}[Sperner \cite{S1928}]\label{NMP} 
Let $n$ and $j$ be two positive integers with $1\leq j\leq n - 1$. For any non-empty $\mathcal{F}\subseteq\binom{[n]}{j}$, 
\[
\frac{|\nabla(\mathcal{F})|}{|\mathcal{F}|}\geq\frac{\binom{n}{j + 1}}{\binom{n}{j}},
\]
where $\nabla(\mathcal{F})=\left\{B\in\binom{[n]}{j+1}:A\subset B\text{ for some }A\in\mathcal{F}\right\}$.
\end{theorem}

\section{Proofs of some basic inequalities}

In this section, we prove technical results, Lemmas \ref{IM1}-\ref{IM3}, which are crucial to our proof of Theorem \ref{TM}. 

For positive integers \(n, s, i\) and \(j\), we define
\begin{align*}
& S(n, s, i, j) = s(n-s+1)-i(j-i), \\ 
& T(n, s, i, j) = i(n-j-s+i+1)+(s-i)(j-i+1). 
\end{align*}	
For \(t+2\leq i\leq s-2\), we  establish the following inequality.
\begin{lemma}\label{IM1}
Let $n, m, k, l, s, t$ and \(i\) be positive integers with \(\min\{m,n\}\geq (t+1)(k-t+1)\), \(k\geq l> t\geq 3\) and \(i+2\leq s \leq i+k-t\). If \(t+2 \leq i\leq k\), then
\begin{align*}
\frac{(m-l+t-i)(n-s-k+i)S(n,s,i,k)S(m,s,s+t-i,l)}{(m-s+1)(n-s+1)T(n,s,i,k)T(m,s,s+t-i,l)}>1,
\end{align*}
where $S(n,s,i,k)$ and $T(n,s,i,j)$ are defined as above.
\end{lemma}

\begin{proof} 
Set
\[
h(m,n,s,k,l,i,t)=(n-s-k+i)S(m,s,s+t-i,l)-(m-s+1)T(n,s,i,k)
\]
and
\[
H(m,n,s,k,l,i,t)=(m+t-l-i)S(n,s,i,k)-(n-s+1)T(m,s,s+t-i,l).
\]
It suffices to show that \(h(m,n,s,k,l,i,t)>0\) and \(H(m,n,s,k,l,i,t)>0\). 

Since \(n\geq(t+1)(k-t+1)\) and \(s> i\), we have 
\begin{align*}
\frac{\partial h}{\partial l}(m, n, s, k, l, i, t)&=-(n-s-k+i)(s+t-i)\\
&\leq -(s+t-i)\big((t+1)(k-t+1)-s-k+i\big)\\
&=-(s+t-i)\big((t-1)(k-t)+(i+k-s-t)+1\big)\\
&<0,
\end{align*}
where the last inequality holds due to the condition \(i+k\geq s+t\). 
Thus, \(h(m, n, s, k, l, i, t)\) is a decreasing function in $l$. It follows that 
\[h(m,n,s,k,l,i,t)\geq h(m,n,s,k,k,i,t).\]
Moreover, since \(n\geq(t+1)(k-t+1)\), \(k\geq i\geq t+2\) and \(i+k\geq s+t\), it follows that 
\begin{equation}\label{3-0}
\begin{aligned}
\min\{m,n\}-s+1 &\geq \min\{m,n\}-k-i+t+1 \\
&\geq (t+1)(k-t+1)-k-i+t+1\\
&=(t-1)(k-t)+(k-i)+2\\
&>0.
\end{aligned}
\end{equation}
This implies that 
\begin{align*}
\frac{\partial H}{\partial l}(m, n, s, k, l, i, t)&=-2(i-t)(n-s+1)+(k-i)i\\
&< -2(i-t)\big((t-1)(k-t)+(k-i)\big)+(k-i)i\\
&= (t-1)(k-i)-2(i-t)(t-1)(k-t)-(k-i)(i-t-1)\\
&< (t-1)(k-t)-2(i-t)(t-1)(k-t)-(k-i)(i-t-1)\\
&= -(2i-2t-1)(t-1)(k-t)-(k-i)(i-t-1)\\
&<0,
\end{align*}
where the last inequality holds since \(i\geq t+2.\) Thus, \(H(m, n, s, k, l, i, t)\) is also a decreasing function in $l$. It follows that 
\[H(m,n,s,k,l,i,t)\geq H(m,n,s,k,k,i,t).\]
Therefore, to complete the proof, it suffices to show that
\begin{equation}\label{3-01}
\begin{aligned}
h(m, n, s, k, k, i, t)>0\ {\rm and}\ H(m, n, s, k, k, i, t)>0,
\end{aligned}
\end{equation}
under the assumptions that \(\min\{m, n\} \geq(t+1)(k-t+1)\) and \(\max\{t+2, s+t-k\} \leq i \leq \max\{s-2, k\}\). 

Since \(k> t\geq 3\) and \( m\geq (t+1)(k-t+1)\), we have  
\begin{align*}
\frac{{\partial}^2 h}{\partial s^2}(m, n, s, k, k, i, t)&=-2(m-2k+2t)\\
&\leq -2\big((t+1)(k-t+1)-2k+2t\big)\\
&= -2(t-1)(k-t-1)-4t\\
&<0. 
\end{align*}
It follows that $h(m, n, s, k, k, i, t)$ is a concave function in $s$. Similarly, since 
\begin{align*}
\frac{{\partial}^2 H}{\partial s^2}(m, n, s, k, k, i, t)=-2(i-t-1)<0, 
\end{align*}
$H(m, n, s, k, k, i, t)$ is also concave in $s$. 

If $i\leq \frac{s+t}{2}$, then $2i-t\leq s\leq i+k-t$ by assumption. Since both functions are concave in \(s\) over the interval \([2i-t, i+k-t]\), their minimum values must occur at the endpoints. Note that  \(h(n,m,2i-t,k,k,i,t)=H(m,n,2i-t,k,k,i,t)\). Consequently, to ensure that they are positive over the entire interval, it suffices to verify:
\begin{equation}\label{h>0,H>0}
\left \{	\begin{array}{l}
 h(m, n, k+i-t, k, k, i, t)>0, \\ H(m, n, k+i-t, k, k, i, t)>0,\ {\rm and}\\ h(m, n, 2i-t, k, k, i, t)>0.
	\end{array}
\right .	
\end{equation}
If $i>\frac{s+t}{2}$, then $\frac{s+t}{2}< i\leq \min\{k,s-2\}$. Since \(h(n,m,s,k,k,s+t-i,t)=H(m,n,s,k,k,i,t)\), the inequality 
\[h(m, n, s, k, k, i, t)>0\ {\rm for}\ \min\{m, n\} \geq(t+1)(k-t+1),\  {\rm and}\ \frac{s+t}{2} \leq i \leq \min\{k, s-2\}\]
is equivalent to
\[H(n, m, s, k, k, i, t)>0\ {\rm for}\ \min\{m, n\} \geq(t+1)(k-t+1),\  {\rm and}\ \max\{t+2, s+t-k\} \leq i \leq \frac{s+t}{2}.\]  Similarly, the inequality
\[H(m, n, s, k, k, i, t)>0\ {\rm for}\ \min\{m, n\} \geq(t+1)(k-t+1),\  {\rm and}\ \frac{s+t}{2} \leq i \leq \min\{k, s-2\}\]
is equivalent to
\[h(n, m, s, k, k, i, t)>0\ {\rm for}\ \min\{m, n\} \geq(t+1)(k-t+1),\  {\rm and}\ \max\{t+2, s+t-k\} \leq i \leq \frac{s+t}{2}.\]
Therefore, to ensure that they are positive over the entire interval, it also suffices to verify (\ref{h>0,H>0}).

Firstly, we show that \[h(m, n, k+i-t, k, k, i, t)>0.\]
Since \(k\geq i\geq t+2\), we have \(2(k-t-1)\geq (k-t).\) Combining this with the assumption that \( m\geq (t+1)(k-t+1)\), we obtain 
\begin{equation}\label{3-1}
\begin{aligned}
& \min\{m,n\}-2i+t+1 \\
& = \min\{m,n\}-k-i+t+1+(k-i)\\
& = \min\{m,n\}-2k+t+1+2(k-i)\\
& \geq (t+1)(k-t+1)-2k+t+1+2(k-i)\\
& = (t-1)(k-t)+2(k-i)+2 \\
& = (t+1)(k-i)+(t-1)(i-t)+2 \\
&\geq t(k-i+2)+(k-i).
\end{aligned}
\end{equation}
Then we have \(\min\{m,n\}-k-i+t+1\geq t(k-i+2) >0\), and this implies that 
\begin{align*}
\vspace{0.1cm}&h(m, n, k+i-t, k, k, i, t)\\
\vspace{0.1cm}=&(n-2k+t)(k+i-t)(m-k-i+t+1)\\
\vspace{0.1cm}&-(m-k-i+t+1)\big(i(n-2k+t+1)+(k-t)(k-i+1)\big)\\
\vspace{0.1cm}=&(m-k-i+t+1)\big((k-t)(n-3k+t-1)+(k-t-1)i\big) \\
\vspace{0.1cm}\geq &(m-k-i+t+1)\big((k-t)((t+1)(k-t+1)-3k+t-1)+(k-t-1)i\big) \\
\vspace{0.1cm}=&(m-k-i+t+1)\big((k-t)(t-2)(k-t-1)+(k-t-1)i-2(k-t)\big) \\
>&0,
\end{align*} 
where the last inequality holds since \(i\geq 5\) and  \((k-t-1)i-2(k-t)\geq 4(k-t-1)-2(k-t)>0.\)

Secondly, we show that \[H(m, n, k+i-t, k, k, i, t)>0.\] 
Given that \(i\geq t+2\) and \(t\geq 3\), we obtain 
\begin{equation}\label{3-2} 
\begin{aligned}
\vspace{0.1cm}&t(k-i+2)(i-t)-i(k-i)-2(i-t)-k\\
\vspace{0.1cm}=&(k-i)((t-1)(i-t-1)-2)+(2t-3)(i-t)-t\\
\vspace{0.1cm}\geq & (k-i)((t-1)(i-t-1)-2)+3(t-2)\\
>&0.
\end{aligned}
\end{equation}
This further implies
\begin{align*}
\vspace{0.1cm}&H(m, n, k+i-t, k, k, i, t)\\
\vspace{0.1cm}=&(n-k-i+t+1)\big((m-k-i+t-1)(i-t)-k\big)-(m-k-i+t)i(k-i)\\
\vspace{0.1cm}\geq & t(k-i+2)\big((m-k-i+t-1)(i-t)-k\big)-(m-k-i+t)i(k-i)\\
\vspace{0.1cm}=& \big(t(k-i+2)(i-t)-i(k-i)\big)(m-k-i+t-1)-kt(k-i+2)-i(k-i)\\
\vspace{0.1cm}\geq & \big(t(k-i+2)(i-t)-i(k-i)\big)(t(k-i+2)-2)-kt(k-i+2)-i(k-i)\\
\vspace{0.1cm}= & t(k-i+2)\big(t(k-i+2)(i-t)-i(k-i)-2(i-t)-k\big)+i(k-i)\\
>&0, 
\end{align*} 
where the first inequality follows from (\ref{3-1}) and 
\begin{align*}
\vspace{0.1cm}(m-k-i+t-1)(i-t)-k \geq & 2(m-k-i+t-1)-k\\
\vspace{0.1cm}\geq & 2(t+1)(k-t+1)-5k +2t -2\\
\vspace{0.1cm}= & (2t-3)(k-t-1)+ (t-3)\\
\geq&0, 
\end{align*}
 the second inequality follows from (\ref{3-1}) and (\ref{3-2}), and the last inequality follow from (\ref{3-2}).  

Finally, we show that \[h(m, n, 2i-t, k, k, i, t)>0.\] 
Since \(k\geq i\geq t+2\geq 5\) and \( n\geq (t+1)(k-t+1)\), we have 
\begin{equation}\label{3-4}
\begin{aligned} 
(i-t)(n-2k+t-2)-t &\geq (i-t)((t-1)(k-t)-1)-t\\
& \geq  (i-t)((t-1)(i-t)-1)-t\\
& \geq 2(2t-3)-t\\
& =3(t-2)\\
&>0.
\end{aligned}
\end{equation}
If \(k=i\), by (\ref{3-1}) and (\ref{3-4}), we have
\[h(m, n, 2i-t, k, k, i, t)= (m-2i+t+1)\big((i-t)(n-2k+t-2)-t\big)> 0. \]
If \(k\geq i+1\), then we have 
\begin{equation}\label{3-5} 
\begin{aligned}
\vspace{0.1cm}&(it-t^2-t)\big((t-1)(k-t)-1\big)-i(k-i+2)-(t+1)t+\big((t-1)(i-t)+2 \big)(i-t)(t-1)\\
\vspace{0.1cm}&=(k-i)\big((it-t^2-t)(t-1)-i\big)-i(t+2)+(2it-2t^2-i+2)(t-1)(i-t)\\
\vspace{0.1cm}&\geq (it-t^2-t)(t-1)-i(t+3)+2(2it-2t^2-i+2)(t-1)\\
\vspace{0.1cm}&=i(5t^2-8t-1)-(5t^2+t-4)(t-1)\\ 
\vspace{0.1cm}&\geq (t+2)(5t^2-8t-1)-(5t^2+t-4)(t-1)\\
\vspace{0.1cm}&= 6(t-3)(t+1)+12\\
&>0,
\end{aligned}
\end{equation}
where the first and second inequalities hold because of \(i\geq t+2\). 
Therefore, 
\begin{align*}
\vspace{0.1cm}&h(m, n, 2i-t, k, k, i, t)\\
\vspace{0.1cm}=&(m-2i+t+1)\big((i-t)(n-2k+t-2)-t\big)-i(k-i)(n-k-i+t)\\
\vspace{0.1cm}\geq & \big((t+1)(k-i)+(t-1)(i-t)+2 \big)\big((i-t)(n-2k+t-2)-t\big)-i(k-i)(n-k-i+t)\\
\vspace{0.1cm}= & \big(it-t^2-t \big)(k-i)(n-2k+t-2)+\big((t-1)(i-t)+2 \big)(i-t)(n-2k+t-2)\\
\vspace{0.1cm}&-\big((t+1)(k-i)+(t-1)(i-t)+2 \big)t-i(k-i)(k-i+2)\\
\vspace{0.1cm}\geq & \big(it-t^2-t \big)(k-i)\big((t-1)(k-t)-1\big)+\big((t-1)(i-t)+2 \big)(i-t)\big((t-1)(k-t)-1\big)\\
\vspace{0.1cm}&-\big((t+1)(k-i)+(t-1)(i-t)+2 \big)t-i(k-i)(k-i+2)\\
\vspace{0.1cm}= & (k-i)[(it-t^2-t)\big((t-1)(k-t)-1\big)-i(k-i+2)-(t+1)t+\big((t-1)(i-t)+2 \big)(i-t)(t-1)]\\
\vspace{0.1cm}&+\big((t-1)(i-t)+2 \big)[(i-t)\big((t-1)(i-t)-1\big)-t]\\
>&0,
\end{align*} 
where the first inequality follows from (\ref{3-1}) and (\ref{3-4}),
the second inequality follows from (\ref{3-1}), 
 and the last inequality follows from (\ref{3-4}) and (\ref{3-5}). 
This completes the proof.  \qed
\end{proof}

For the case \(i=s-1\), we obtain the following inequality.
\begin{lemma}\label{IM2}
Let \(n, m, k, l, s, t\) be positive integers with \(\min\{m,n\}\geq (t+1)(k-t+1)\), \(k\geq l\geq t+1\geq 4\) and \(t+3\leq s \leq k+1\). Then
\begin{align*}
&(t+1)(k-s+2)(m-l+t-s+2)\\
&\hspace{3cm}-(s-1-t)\big((m-l+t-s+1)(n-k-1)-(k-s+2)(l-t)\big)<0.
\end{align*}
\end{lemma}

\begin{proof} Let
\begin{align*}
&g(m, n, k, l, s, t)=(t+1)(k-s+2)(m-l+t-s+2)\\
&\hspace{3cm}-(s-1-t)\big((m-l+t-s+1)(n-k-1)-(k-s+2)(l-t)\big).
\end{align*}
Given that \(n\geq (t+1)(k-t+1)\) and \(k\geq l\), we obtain 
\begin{align*}
\frac{{\partial}^2 g}{\partial s^2}=2(n-k-l+2t)\geq 2((t-1)(k-t-1)+2t)>0.
\end{align*}
This implies that \(g(m, n, k, l, s, t)\) is a convex function in \(s\). Therefore, it suffices to show that 
\[g(m, n, k, l, t+3, t)<0\ {\rm and}\ g(m, n, k, l, k+1, t)<0.\]

Firstly, we show that 
\[g(m, n, k, l, t+3, t)<0.\]
Since \(\min\{m,n\}\geq (t+1)(k-t+1)\), we have
\begin{equation}\label{3-7} 
\begin{aligned}
(t+1)(l-t-1)-2(\min\{m,n\}-l-1) & \leq(t+1)(k-t-1)-2(\min\{m,n\}-k-1) \\
& \leq (t+1)(k-t-1)-2((t+1)(k-t+1)-k-1)\\
&=-(t-1)(k-t+1)-2\\
&<0.
\end{aligned}
\end{equation}
This implies that 
\begin{align*}
&g(m, n, k, l, t+3, t)\\
&=(t+1)(k-t-1)(m-l-1)-2\big((m-l-2)(n-k-1)-(k-t-1)(l-t)\big)\\
&=\big((t+1)(k-t-1)-2(n-k-1)\big)(m-l-1)+2(k-t-1)(l-t)+2(n-k-1)\\
&\leq \big((t+1)(k-t-1)-2(n-k-1)\big)t(k-t)+2(k-t-1)(l-t)+2(n-k-1)\\
&=(t+1)(k-t-1)t(k-t)+2(k-t-1)(l-t)-2(t(k-t)-1)(n-k-1)\\
&\leq (t+1)(k-t-1)t(k-t)+2(k-t-1)(l-t)-2(t(k-t)-1)t(k-t)\\
&= \big((t+1)(k-t-1)-2t(k-t)\big)t(k-t)+2(k-t-1)(l-t)+2t(k-t)\\
&= -(t-1)(k-t+1)t(k-t)+2(k-t-1)(l-t)\\
&\leq -(k-t+1)(k-t)\big((t-1)t-2\big)\\
&<0,
\end{align*}
where the first and second inequalities follow from (\ref{3-7}) and 
\[ \min\{n, m\}-l-1\geq \min\{n, m\}-k-1\geq (t+1)(k-t+1)-k-1\geq t(k-t).\]

Finally, we show that 
\[g(m, n, k, l, k+1, t)<0.\]
Given that \(\min\{m,n\}\geq (t+1)(k-t+1)\) and \(k\geq t+2\), it follows that \\
\begin{equation}\label{3-8} 
\begin{aligned}
(t+1)-(\min\{m,n\}-l-1)(k-t) & \leq (t+1)-(\min\{m,n\}-k-1)(k-t) \\
& \leq (t+1)-(k-t)((t+1)(k-t+1)-k-1)\\
&= (t+1)-(k-t)^2t\\
&\leq (t+1)-2(tk-t^2)\\
&<-t(k-t).
\end{aligned}
\end{equation}
This implies that 
\begin{align*}
&g(m, n, k, l, k+1, t)\\
&=(t+1)(m-l+t-k+1)-(k-t)\big((m-l+t-k)(n-k-1)-(l-t)\big)\\
&=\big((t+1)-(n-k-1)(k-t)\big)(m-l+t-k)+(k-t)(l-t)+t+1\\
&\leq \big((t+1)-(n-k-1)(k-t)\big)((t-1)(k-t)+1)+(k-t)(l-t)+t+1\\
&\leq -t(k-t)((t-1)(k-t)+1)+(k-t)(k-t)+t+1\\
&=-(k-t)^2(t(t-1)-1)-t(k-t-2)-(t-1)\\
&<0, 
\end{align*}
where the first and second inequalities follow from (\ref{3-1}), (\ref{3-8}) and \(k\geq l\).
This completes the proof. \qed
\end{proof}

For the case \(i=t+1\), we have the following inequality.
\begin{lemma}\label{IM3}
Let \(n, m, k, l, s, t\) be positive integers with \(\min\{m,n\}\geq (t+1)(k-t+1)\) and \(k\geq l\geq s-1\geq t+2\geq 5\). Then
\begin{align*}
&(t+1)(n-s-k+t+2)(l-s+2)\\
&\hspace{3cm}-(s-1-t)\big((m-l-1)(n-s-k+t+1)-(k-t)(l-s+2)\big)<0.
\end{align*}
\end{lemma}

\begin{proof} Let
\begin{align*}
f(m, n, k, l, s, t)&=(t+1)(n-s-k+t+2)(l-s+2)\\
&\hspace{0.5cm}-(s-1-t)\big((m-l-1)(n-s-k+t+1)-(k-t)(l-s+2)\big).
\end{align*}
Given that \(n\geq (t+1)(k-t+1)\) and \(k\geq l\), we obtain 
\begin{align*}
\frac{{\partial}^2 f}{\partial s^2}=2(m-k-l+2t)\geq 2((t-1)(k-t-1)+2t)>0.
\end{align*}
This implies that \(f(m, n, k, l, s, t)\) is a convex function in \(s\). Therefore, it suffices to show that \(f(m, n, k, l, t+3, t)<0\) and \(f(m, n, k, l, k+1, t)<0\).

Firstly, we show that \[f(m, n, k, l, t+3, t)<0.\]

Since \(\min\{m,n\}\geq (t+1)(k-t+1)\) and \(k\geq l\), we have 
\begin{align*}
&f(m, n, k, l, t+3, t)\\
&=(t+1)(l-t-1)(n-k-1)-2\big((m-l-1)(n-k-2)-(l-t-1)(k-t)\big)\\
&=\big((t+1)(l-t-1)-2(m-l-1)\big)(n-k-1)+2(l-t-1)(k-t)+2(m-l-1)\\
&\leq \big((t+1)(l-t-1)-2(m-l-1)\big)t(k-t)+2(l-t-1)(k-t)+2(m-l-1)\\
&=(t+1)(l-t-1)t(k-t)-2\big(t(k-t)-1\big)(m-l-1)+2(l-t-1)(k-t) \\
&\leq (t+1)(k-t-1)t(k-t)-2\big(t(k-t)-1\big)t(k-t)+2(l-t-1)(k-t)\\
&= -(t-1)(k-t+1)t(k-t)+2(l-t-1)(k-t)\\
&\leq -(k-t+1)(k-t)\big((t-1)t-2\big)\\
&<0, 
\end{align*}
where the first and second inequalities follow from (\ref{3-1}) and  (\ref{3-7}).

Finally, we show that \[f(m, n, k, l, k+1, t)<0.\]

Given that \(\min\{m,n\}\geq (t+1)(k-t+1)\) and \(k\geq t+2\), it follows that 
\begin{align*}
&f(m, n, k, l, k+1, t)\\
&=(t+1)(n-2k+t+1)(l-k+1)-(k-t)\big((m-l-1)(n-2k+t)-(k-t)(l-k+1)\big)\\
&=-(t+1)(n-2k+t+1)(k-l)-(k-t)^2(k-l)\\
&\hspace{0.5cm}-((k-t)(m-l-1)-t-1)(n-2k+t)+t+1+(k-t)^2\\
&\leq -((k-t)(m-l-1)-t-1)((t-1)(k-t)+1)+t+1+(k-t)^2\\
&\leq -t(k-t)((t-1)(k-t)+1)+(k-t)^2+t+1\\
&=-(k-t)^2(t(t-1)-1)-t(k-t-2)-(t-1)\\
&<0,
\end{align*}
where the first inequality follows from (\ref{3-1}), 
and the second inequality follows from (\ref{3-8}). This completes the proof. \qed
\end{proof}

\section{Proof of Theorem \ref{TM}}
In this section, we prove Theorem \ref{stl} and Theorem \ref{TM}. Frankl, Lee,  Siggers and Tokushige \cite{FLST2014} established the following characterization of the extremal structures for the uniform cross-\(t\)-intersecting families.

\begin{theorem}[Frankl, Lee,  Siggers and Tokushige \cite{FLST2014}]
Let $n$, $k$, $t$, $r$ be positive integers such that $t\geq2$, $k - t\geq r$ and $n\geq2k - t+2$. If $\mathcal{A}\subseteq\binom{[n]}{k}$ and $\mathcal{B}\subseteq\binom{[n]}{k}$ are cross-$t$-intersecting, and $S_{ij}(\mathcal{A})=S_{ij}(\mathcal{B})=\mathcal{F}(n,k,t,r)$, then $\mathcal{A}=\mathcal{B}\cong\mathcal{F}(n,k,t,r)$.
\end{theorem}

Borg \cite{B2016} obtained the following characterization of the extremal structures for the non-uniform cross-\(t\)-intersecting families. 

\begin{theorem}[Borg \cite{B2016}]
Let $m, n, k, l, t, r$ be positive integers such that $1 \leq t \leq l \leq k \leq n, l\leq m$ and $\min\{m,n\}\geq (t+4)(k-t)+l-1$. Let $\mathcal{A}\subseteq\binom{[n]}{k}$ and $\mathcal{B}\subseteq\binom{[m]}{l}$ be  cross-$t$-intersecting. Suppose that \(S_{ij}(\mathcal{A})=\left\{A \in \binom{[n]}{k} : T\subset A\right\}\) and \(S_{ij}(\mathcal{B})=\left\{B \in \binom{[m]}{l} : T\subset B \right\}\) for some \(t\)-element subset of \([\min\{m,n\}]\). Then \(\mathcal{A}=\left\{A \in \binom{[n]}{k} :  T' \subset A \right\}\) and \(\mathcal{B}=\left\{B \in \binom{[m]}{l} :  T' \subset B  \right\}\) for some \(t\)-element subset \(T'\subset [\min\{m,n\}]\). 
\end{theorem}

We generalize the above characterization of the extremal structures for the cross-\(t\)-intersecting families by starting with two technical lemmas.

\begin{lemma}\label{stl1}
Let \(m, n, k, l, t, r\) be positive integers satisfying \(k\geq l \geq t+r\) and \(\min\{m, n\}\geq k+l-t+2\). Let \(T\) be a \((t+2r)\)-element subset of \([\min\{m,n\}]\) and fix \(j\in [\min\{m,n\}]\setminus T\). For any subsets \(A\in \binom{[n]\setminus (T\cup \{j\})}{k-r-t}\) and \(B \in \binom{[m]\setminus (T\cup \{j\})}{l-r-t}\) with \(|A\cap B|=w\), there exist subsets \(A_1, A_2, \ldots, A_{w+1} \in \binom{[n]\setminus (T\cup \{j\})}{k-r-t}\) and \(B_1, B_2, \ldots, B_{w+1} \in \binom{[m]\setminus (T\cup \{j\})}{l-r-t}\) such that
\begin{itemize}
\item[{\rm (1)}] \(A_i \cap B_i=\emptyset\) for \(i=1,\ldots, w+1\);
\item[{\rm (2)}] \(A_{i+1} \cap B_i=\emptyset\) for \(i=1,\ldots, w\);
\item[{\rm (3)}] \(A_1=A, B_{w+1}=B\).
\end{itemize}
\end{lemma}

\begin{proof}
Let \(A\cap B=\{a_{1}, a_{2}, \ldots, a_{w}\}\). Since \(\min\{m, n\} \geq k+l-t+2\), we have 
\[|[\min\{m, n\}]\setminus (T\cup \{j\} \cup A)|=\min\{m, n\}-(t+2r)-1-(k-t-r)\geq (l-t-r)+1 \geq w+1.\]
Take a \((w+1)\)-element subset from \([\min\{m, n\}]\setminus (T\cup \{j\} \cup A)\), denoted as \(\{b_0, b_1, b_2, \ldots, b_{w}\}.\) Set \(B_1=(B\setminus A) \cup \{b_1, b_2, \ldots, b_{w}\}\). Then, recursively define subsets \(B_2, \ldots, B_{w+1} \in \binom{[m]\setminus (T\cup \{j\})}{l-r-t}\) and \( A_2, \ldots, A_{w+1} \in \binom{[n]\setminus (T\cup \{j\})}{k-r-t}\) as follows: 
\begin{align*}
A_s=\left(A_{s-1}\setminus \{a_{s-1}\}\right) \cup \{b_{s-2}\}\ {\rm and}\ \  B_s=\left(B_{s-1}\setminus \{b_{s-1}\}\right) \cup \{a_{s-1}\}
\end{align*}
for \(s=2,\ldots, w+1\). Clearly, \(A_1, A_2, \ldots, A_{w+1}\) and \(B_1, B_2, \ldots, B_{w+1}\) satisfy the required properties. This completes the proof. \qed
\end{proof}

\begin{lemma}\label{stl2}
Let \(n, t', r\) be positive integers. Let \(T\) be a \((t'+2r)\)-element subset of \([n]\). For any distinct subsets \(A, B\in \binom{T}{r+t'}\), there exist subsets \(A_1, A_2, \ldots, A_{2w} \in \binom{T}{r+t'}\) such that
\begin{itemize}
\item[{\rm (1)}] \(|A_i \cap A_{i+1}|=t'\) for \(i=1,\ldots, 2w-1\);
\item[{\rm (2)}] \(A_1=B, A_{2w}=A\),
\end{itemize}
where  \(w=\left\lfloor\frac{|A \cap B|}{t'}\right\rfloor+1.\)
\end{lemma}

\begin{proof}
Let \(|A\cap B|=t'+b\). Denote \(A\cap B=\{a_{1}, a_{2}, \ldots, a_{t'+b}\}\) and let \(T=\{a_{1}, a_{2}, \ldots, a_{t'+2r}\}\). Without loss of generality, we may assume \((t'+r)\)-element subsets \( A=\{a_{1}, a_{2}, \ldots, a_{t'+r}\}\) and \(B=\{a_{1}, a_{2}, \ldots, a_{t'+b}\}\cup \{ a_{t'+r+1}, a_{t'+r+2}, \ldots, a_{t'+2r-b} \}.\) Set
\(b=mt'+b'\), where \(0\leq b'<t'\). Note that \(r> b\) and \(w=m+2\). For each $s=0,1,\ldots,m$, define the following subsets 
\begin{align*}
A_{2s+1}=&\{a_{1}, a_{2}, \ldots, a_{b+t'(1-s)}\} \cup  \{a_{r+t'+1}, a_{r+t'+2}, \ldots, a_{(s+1)t'+2r-b}\},\ {\rm and}\\
A_{2s+2}=&\{a_{b-st'+1}, a_{b-st'+2}, \ldots, a_{t'+r}\} \cup  \{a_{(s+1)t'+2r-b+1}, a_{(s+1)t'+2r-b+2}, \ldots, a_{t'+2r}\}. 
\end{align*}
Set
\begin{align*}
A_{2m+3}=&\{a_{1}, a_{2}, \ldots, a_{b-t'm}\} \cup  \{a_{b+r-t'm+1}, a_{b+r-t'm+2}, \ldots, a_{t'+2r}\},\ {\rm and}\\
A_{2m+4}=&\{a_{1}, a_{2}, \ldots, a_{t'+r}\}.
\end{align*}
By the definition of \(A_s\), we have 
\[A_{1}=\{a_{1}, a_{2}, \ldots, a_{b+t'}\} \cup  \{a_{r+t'+1}, a_{r+t'+2}, \ldots, a_{t'+2r-b}\}=B\ {\rm and}\ A_{2m+4}=A.\]
Since \(r> b\), each \(A_s\) is a \((t'+r)\)-subset of \(T\). It is easy to see that
\begin{align*}
|A_{2s+1} \cap A_{2s+2}|&=|\{a_{b-st'+1}, a_{b-st'+2}, \ldots, a_{b+t'(1-s)}\} |=t' \ {\rm for }\ s=0,1,\ldots, m,\ {\rm and}\\
|A_{2m+3}\cap A_{2m+4}|&=|\{a_{1}, a_{2}, \ldots, a_{b-t'm}\} \cup  \{a_{b+r-t'm+1}, a_{b+r-t'm+2}, \ldots, a_{t'+r}\}|=t'.
\end{align*}
Moreover, since \(mt' \leq b< (m+1)t'\) and \(r> b\), we have 
\begin{align*}
& |A_{2s} \cap A_{2s+1}|=|\{a_{st'+2r-b+1}, a_{st'+2r-b+2}, \ldots, a_{(s+1)t'+2r-b}\}|=t' \ {\rm for}\ s=1,\ldots, m-1,\ {\rm and}\\
& |A_{2m+2}\cap A_{2m+3}|=|\{a_{b+r-mt'+1}, a_{b+r-mt'+2}, \ldots, a_{t'+r}\} \cup  \{a_{(m+1)t'+2r-b+1}, \ldots, a_{t'+2r}\}|=t'.
\end{align*}
This completes the proof. \qed
\end{proof}

Now, we present the characterization of the general extremal structures for the cross-\(t\)-intersecting families.

\begin{theorem}\label{stl}
Let \(m, n, k, l, t, r\) be positive integers such that \(t\geq 2, k\geq l \geq t+r\) and \(\min\{m, n\}\geq k+l-t+2\). Let \(i,j\in [\max \{ m, n\}] \) with \(i<j\). Let $\mathcal{A}\subseteq\binom{[n]}{k}$ and $\mathcal{B}\subseteq\binom{[m]}{l}$ be cross-$t$-intersecting families. Suppose there exists a \((t+2r)\)-element subset \(T\subset [\min\{m, n\}]\) such that \(S_{ij}(\mathcal{A})=\left\{A \in \binom{[n]}{k} : |A \cap T| \geq t+r \right\}\), \(S_{ij}(\mathcal{B})=\left\{B \in \binom{[m]}{l} : |B \cap T| \geq t+r \right\}\). Then 
 \(\mathcal{A}=\left\{A \in \binom{[n]}{k} : |A \cap T'| \geq t+r \right\}\) and \(\mathcal{B}=\left\{B \in \binom{[m]}{l} : |B \cap T'| \geq t+r \right\}\) for some \((t+2r)\)-element subset \(T'\subset [\min\{m, n\}]\). 
\end{theorem}

\begin{proof}
Based on the relationship between \(i\) and \(j\), we divide the proof into four cases.

{\bf Case 1}: \(i\notin T\) and \(j \notin T\). It is not difficult to see that 
\begin{equation}\label{MTHS}
\begin{aligned}
S_{ij}(\mathcal{A})&=\left\{A \in \binom{[n]}{k} : |A \cap T| \geq t+r \right\}=\bigcup\limits_{V\subseteq T,\atop t+2r \geq |V|\geq t+r}\left\{A \in \binom{[n]}{k} : V\subseteq A \right\}\\
&=\bigcup\limits_{V\subseteq T,\atop \min\{t+2r, k\} \geq |V|\geq t+r}\left\{A'\cup V : A' \in \binom{[n]\setminus T}{k-|V|} \right\}. 
\end{aligned}
\end{equation}
Given that \(i\notin T\) and \(j \notin T\), by the definition of the shift operation \(S_{ij}\) and (\ref{MTHS}), it follows that \(\mathcal{A}=S_{ij}(\mathcal{A})\). Similarly, we have \(\mathcal{B}=S_{ij}(\mathcal{B})\). Then take \(T'=T\), and the conclusion follows. 

{\bf Case 2}: \(i\in T\) and \(j \in T\). Similarly, it holds that 
{\begin{equation}\label{MTHS1}
\begin{aligned}
S_{ij}(\mathcal{A})&=\left\{A \in \binom{[n]}{k} : |A \cap T| \geq t+r \right\}=\bigcup\limits_{V\subseteq T,\atop t+2r \geq |V|\geq t+r}\left\{A \in \binom{[n]}{k} : V\subseteq A \right\}\\
&=\bigcup\limits_{A' \subseteq [n]\setminus T,\atop k-t-r \geq |A'|\geq \min\{0, k-2r-t\}}\left\{A'\cup V : V \in \binom{T}{k-|A'|} \right\}.
\end{aligned}
\end{equation}}
Given that \(i\in T\) and \(j \in T\), it follows from (\ref{MTHS1}) that \(\mathcal{A}=S_{ij}(\mathcal{A})\). Similarly, we have \(\mathcal{B}=S_{ij}(\mathcal{B})\). Then take \(T'=T\), and the conclusion follows. 

{\bf Case 3}: \(j\in T\) and \(i \notin T\). In this case, we have 
{\small
\begin{align*}
&S_{ij}(\mathcal{A})=\left\{A \in \binom{[n]}{k} : |A \cap T| \geq t+r \right\}=\bigcup\limits_{V\subseteq T\setminus \{j\},\atop \min\{t+2r-1,k\} \geq |V|\geq t+r}\left\{A \in \binom{[n]\setminus \{i,j\}}{k} :  V\subseteq A \right\}\\
&\ \ \ \hspace{2cm}\cup\bigcup\limits_{V\subseteq T\setminus \{j\},\atop\min\{t+2r-1,k-2\} \geq |V|\geq t+r-1}\left\{A \in \binom{[n]}{k} : \{i, j\}\cup V\subseteq A \right\}\\
& \ \ \ \hspace{2cm}\cup\bigcup\limits_{V\subseteq T\setminus \{j\},\atop \min\{t+2r-1,k-1\} \geq |V|\geq t+r}\bigcup\limits_{A' \subseteq [n]\setminus (T\cup \{i\}),\atop  |A'|=k-|V|-1}\left\{A'\cup V \cup \{i\}, A'\cup V \cup \{j\} \right\}\\
& \ \ \ \hspace{2cm}\cup\bigcup\limits_{V\subseteq T\setminus \{j\},\atop |V|=t+r-1} \left\{A'\cup V \cup \{j\} : A' \in \binom{[n]\setminus (T\cup \{i\})}{k-r-t} \right\}.
\end{align*}}
This implies that  
\[A'\cup V \cup \{j\}\in S_{ij}(\mathcal{A})\ \text{ and}\ A'\cup V \cup \{i\} \notin S_{ij}(\mathcal{A})\]
for \(A' \in \binom{[n]\setminus (T\cup \{i\})}{k-r-t}\) and \(V\subseteq T\setminus \{j\}\) with \(|V|=t+r-1\), which contradicts the definition of the shift operation \(S_{ij}\).  Then this case cannot occur.

{\bf Case 4}: \(i\in T\) and \(j \notin T\). In this case, we have 
{\small 
\begin{equation}\label{MTHS2}
\begin{aligned}
&S_{ij}(\mathcal{A})=\left\{A \in \binom{[n]}{k} : |A \cap T| \geq t+r \right\}\\
&=\bigcup\limits_{V\subseteq T\setminus \{i\},\atop \min\{t+2r-1,k\} \geq |V|\geq t+r}\left\{A \in \binom{[n]\setminus \{i,j\}}{k} :  V\subseteq A \right\}\\
& \ \ \ \cup\bigcup\limits_{V\subseteq T\setminus \{i\},\atop \min\{t+2r-1,k-2\} \geq |V|\geq t+r-1}\left\{A \in \binom{[n]}{k} : \{i, j\}\cup V\subseteq A \right\}\\
& \ \ \ \cup\bigcup\limits_{V\subseteq T\setminus \{i\},\atop \min\{t+2r-1,k-1\} \geq |V|\geq t+r}\bigcup\limits_{A' \subseteq [n]\setminus (T\cup \{j\}),\atop  |A'|=k-|V|-1}\left\{A'\cup V \cup \{i\}, A'\cup V \cup \{j\} \right\}\\
& \ \ \ \cup\bigcup\limits_{V\subseteq T\setminus \{i\},\atop |V|=t+r-1} \left\{A'\cup V \cup \{i\} : A' \in \binom{[n]\setminus (T\cup \{j\})}{k-r-t} \right\}.
\end{aligned}
\end{equation}}
It is easy to see that for any set \(A_1\) with 
\[S_{ij}(A_1) \in S_{ij}(\mathcal{A})\setminus \left\{A'\cup V \cup \{i\} : A' \in \binom{[n]\setminus (T\cup \{j\})}{k-r-t}, V\subseteq T\setminus \{i\}, |V|=t+r-1 \right\},\]  by the definition of the shift operation \(S_{ij}\), we have \[A_1\in S_{ij}(\mathcal{A})\setminus \left\{A'\cup V \cup \{i\} : A' \in \binom{[n]\setminus (T\cup \{j\})}{k-r-t}, V\subseteq T\setminus \{i\}, |V|=t+r-1 \right\}.\]
Similarly, for \(S_{ij}(B_1) \in S_{ij}(\mathcal{B})\setminus \left\{B'\cup U \cup \{i\} : B' \in \binom{[m]\setminus (T\cup \{j\})}{l-r-t}, U\subseteq T\setminus \{i\}, |U|=t+r-1 \right\},\) we have 
\[B_1\in S_{ij}(\mathcal{B})\setminus \left\{B'\cup U \cup \{i\} : B' \in \binom{[m]\setminus (T\cup \{j\})}{l-r-t}, U\subseteq T\setminus \{i\}, |U|=t+r-1 \right\}.\]

Then we focus on the set \(A\) with \[S_{ij}(A) \in \left\{A'\cup V \cup \{i\} : A' \in \binom{[n]\setminus (T\cup \{j\})}{k-r-t}, V\subseteq T\setminus \{i\}, |V|=t+r-1 \right\}.\]
Based on the relationship between \(\mathcal{A}\) and \(S_{ij}(\mathcal{A})\), we divide the proof into three subcases.

{\bf Subcase 1}:\ \(\mathcal{A}=S_{ij}(\mathcal{A})\). In this subcase, we have  
\[\left\{A'\cup V \cup \{i\} : A' \in \binom{[n]\setminus (T\cup \{j\})}{k-r-t}, V\subseteq T\setminus \{i\}, |V|=t+r-1 \right\}\subset \mathcal{A}.\] 
Since \(|T\setminus \{i\}|=t+2r-1\), for any subset \(U\subset T\setminus \{i\}\) of size \(|U|=t+r-1\), there exists a subset \(V\subset T\setminus \{i\}\) of size \(|V|=t+r-1\) such that \(|V \cap U|=t-1\). Given that \(\min\{m, n\} \geq k+l-t+2\), for any subset \(B'\subset [m]\setminus (T\cup \{j\})\) with \(|B'|=l-r-t\), and for any \((k-r-t)\)-subset \(A'\subset [n]\setminus (T\cup \{j\}\cup B')\), it holds that 
\[|(A'\cup V \cup \{i\}) \cap (B'\cup U \cup \{j\})|=t-1. \]
This contradicts the property of the cross-$t$-intersecting families $\mathcal{A}$ and $\mathcal{B}$. Similar to (\ref{MTHS2}), we have 
\[\bigcup\limits_{U\subseteq T\setminus \{i\},\atop |U|=t+r-1} \left\{B'\cup U \cup \{i\} : B' \in \binom{[m]\setminus (T\cup \{j\})}{l-r-t} \right\}\subseteq S_{ij}( \mathcal{B}).\]
This implies that 
\[B'\cup U \cup \{j\}\notin \mathcal{B}\ \text{and}\ B'\cup U \cup \{i\}\in \mathcal{B}\ \text{ for}\ B' \in \binom{[m]\setminus (T\cup \{j\})}{l-r-t}, U\subseteq T\setminus \{i\}\ \text{and}\ |U|=t+r-1.\] 
Hence, we have \(\mathcal{B}=S_{ij}(\mathcal{B})\).

{\bf Subcase 2}:\ \(\mathcal{B}=S_{ij}(\mathcal{B})\). By an argument analogous to that in Subcase 1, we obtain \(\mathcal{A}=S_{ij}(\mathcal{A})\). 

{\bf Subcase 3}:\ \(\mathcal{A}\not=S_{ij}(\mathcal{A})\) and \(\mathcal{B}\not=S_{ij}(\mathcal{B})\). In this subcase, we have \(1\leq i< j\leq \min\{m,n\}\). 
By (\ref{MTHS2}), there exist a subset \(A_1 \in \binom{[n]\setminus (T\cup \{j\})}{k-r-t}\) and  a subset \(V_1\subseteq T\setminus \{i\}\) with \(|V_1|=t+r-1\) such that \(A_1\cup V_1 \cup \{j\}\in \mathcal{A}.\) For any fixed \((l-t-r)\)-subset \(B'\) of \([m]\setminus (T\cup \{j\})\), and any fixed \((t+r-1)\)-subset \(U_1\) of \(T\setminus \{i\}\), we proceed to show that
\[B'\cup U_1 \cup \{i\}\notin \mathcal{B}\ \text{ and}\ B'\cup U_1 \cup \{j\}\in \mathcal{B}.\]
Applying Lemma \ref{stl1} with \(A=A_1\) and \(B=B'\), there exist subsets \(A_2, \ldots, A_{w+1} \in \binom{[n]\setminus (T\cup \{j\})}{k-r-t}\) and \(B_1, B_2, \ldots, B_{w+1} \in \binom{[m]\setminus (T\cup \{j\})}{l-r-t}\) satisfying
\[ A_z \cap B_z=\emptyset, A_{z+1} \cap B_z=\emptyset\ \text {for}\ z=1,\ldots, w+1\ \text{and}\ B_{w+1}=B'.\]
For any subsets \(V_1, U_1\) of \(T\setminus \{i\}\), by Lemma \ref{stl2} with \(t'=t-1\), there exist \((r+t-1)\)-subsets \(V_2, V_3, \ldots, V_{2w'+2}\) of \(T\setminus \{i\}\) such that 
\[V_{2w'+2}=U_1\ \text{ and}\ |V_z\cap V_{z+1}|=t-1\ \text{for}\ z=1, \ldots, 2w'+1.\]
By (\ref{MTHS2}), we have
\begin{align*}
\{A\cup V \cup \{i\}, A\cup V \cup \{j\}\}\cap \mathcal{A}\not= \emptyset 
\end{align*}
for any subsets \(A \in \binom{[n]\setminus (T\cup \{j\})}{k-r-t}\), \(V \subseteq T\setminus \{i\}\) with \(|V|=t+r-1.\) Similarly, we have 
\begin{align*}
\{B\cup U \cup \{i\}, B\cup U \cup \{j\}\}\cap \mathcal{B}\not= \emptyset
\end{align*}
for any subsets \(B \in \binom{[m]\setminus (T\cup \{j\})}{l-r-t}\), \(U \subseteq T\setminus \{i\}\) with \(|U|=t+r-1.\) Since \(A_1\cup V_1 \cup \{j\}\in \mathcal{A}\), it follows that \(B_1\cup V_2 \cup \{i\}\notin \mathcal{B}\) because 
\[|(A_1\cup V_1 \cup \{j\})\cap (B_1\cup V_2 \cup \{i\}) |=|V_1 \cap V_2|=t-1.\]
This implies that \(B_1\cup V_2 \cup \{j\}\in \mathcal{B}\) since \(\mathcal{B}\) and \(\mathcal{A}\) are cross-\(t\)-intersecting. By an analogous argument, we can show the following: 
\begin{equation}\label{MTHS3}
\begin{aligned}
& \text{ (i)}\ B_z\cup V_{z'\pm 1} \cup \{i\}\notin \mathcal{B}\ \text{ and}\ B_z\cup V_{z'\pm 1} \cup \{j\}\in \mathcal{B}\ \text{ whenever}\ A_z\cup V_{z'} \cup \{j\}\in \mathcal{A};\\
& {\rm (ii)}\ A_{z+1}\cup V_{z'\pm 1} \cup \{i\}\notin \mathcal{A}\ \text{ and}\ A_{z+1}\cup V_{z'\pm 1} \cup \{j\}\in \mathcal{A}\ \text{ whenever}\ B_z\cup V_{z'} \cup \{j\}\in \mathcal{B}; \\
& {\rm (iii)}\ A_z\cup V_{z'\pm 1} \cup \{i\}\notin \mathcal{A}\ \text{ and}\ A_z\cup V_{z'\pm 1} \cup \{j\}\in \mathcal{A}\ \text{ whenever}\ B_z\cup V_{z'} \cup \{j\}\in \mathcal{B}.
\end{aligned}
\end{equation}
Construct the following \(\max\{2w+2, 2w'+2\}\) subsets:
\begin{align*}
&\text{if\ }w \geq w',\ \text{then\ set}
\begin{cases}
D_{2s-1}=A_s \cup V_{2s-1} \cup \{j\},\ & s=1,2,\ldots, w'+1,\\
D_{2s}=B_s \cup V_{2s} \cup \{j\},\ &s=1,2,\ldots, w'+1,\\
D_{2s-1}=A_s \cup V_{2w'+1} \cup \{j\},\ & s=w'+2, w'+3,\ldots, w+1,\\
D_{2s}=B_s \cup V_{2w'+2} \cup \{j\},\ & s=w'+2, w'+3,\ldots, w+1;
\end{cases}\\
&\text{if\ }w' > w,\ \text{then\ set}
\begin{cases}
D_{2s-1}=A_s \cup V_{2s-1} \cup \{j\},\ & s=1,2,\ldots, w+1,\\
D_{2s}=B_s \cup V_{2s} \cup \{j\},\ &s=1,2,\ldots, w+1,\\
D_{2s-1}=A_{w+1} \cup V_{2s-1} \cup \{j\},\ & s=w+2, w+3,\ldots, w'+1,\\
D_{2s}=B_{w+1} \cup V_{2s} \cup \{j\},\ & s=w+2, w+3,\ldots, w'+1.
\end{cases}
\end{align*}
By repeated application of property (\ref{MTHS3}), we obtain 
\[D_{2s-1}\in \mathcal{A}, D_{2s}\in \mathcal{B}, (D_{2s-1}\cup \{i\})\setminus \{j\}\notin \mathcal{A}\ \text{and}\  (D_{2s}\cup \{i\})\setminus \{j\}\in \mathcal{B}\]
for \(s=1,2,\ldots, \max\{w+1, w'+1\}\). Consequently, for any subset \(B' \in \binom{[m]\setminus (T\cup \{j\})}{l-r-t} \) and any subset \(U_1\subseteq T\setminus \{i\}\) with \(|U_1|=t+r-1\), it holds that  
\[B'\cup U_1 \cup \{i\}=(D_{\max\{2w+2, 2w'+2\}}\cup \{i\})\setminus \{j\} \notin \mathcal{B}\ \text{ and}\ B'\cup U_1 \cup \{j\}=D_{\max\{2w+2, 2w'+2\}} \in \mathcal{B}. \]
This implies that {\small
\begin{align*}
\mathcal{B}&=\bigcup\limits_{U\subseteq T\setminus \{i\},\atop \min\{t+2r-1,l-2\} \geq |U|\geq t+r}\left\{B \in \binom{[m]\setminus \{i, j\}}{l} :  U\subseteq B \right\}\\
& \ \ \ \cup\bigcup\limits_{U\subseteq T\setminus \{i\},\atop \min\{t+2r-1,l-2\} \geq |U|\geq t+r-1}\left\{B \in \binom{[m]}{l} : \{i, j\}\cup U\subseteq B \right\}\\
& \ \ \ \cup\bigcup\limits_{U\subseteq T\setminus \{i\},\atop \min\{t+2r-1,l-1\} \geq |U|\geq t+r}\bigcup\limits_{B' \subseteq [n]\setminus (T\cup \{j\}),\atop  |B'|=l-|U|-1}\left\{B'\cup U \cup \{i\}, B'\cup U \cup \{j\} \right\}
\end{align*} }
\begin{align*}
& \ \ \ \cup\bigcup\limits_{U\subseteq T\setminus \{i\},\atop |U|=t+r-1} \left\{B'\cup U \cup \{j\} : B' \in \binom{[m]\setminus (T\cup \{j\})}{l-r-t} \right\}\\
&=\left\{B \in \binom{[m]}{l} : |B \cap (T\cup\{j\}\setminus \{i\})| \geq t+r \right\}.
\end{align*} 
Hence, we have 
\[\mathcal{B}=\left\{B \in \binom{[m]}{l} : |B \cap T'| \geq t+r \right\}\ {\rm and}\  
\mathcal{A}=\left\{A \in \binom{[n]}{k} : |A \cap T'| \geq t+r \right\},\]
where \(T'=T\cup\{j\}\setminus \{i\}\). This completes the proof. \qed
\end{proof}

We now start to prove the main theorem.

\noindent \textbf{Proof of Theorem \ref{TM}}. By Fact \ref{lcl} and Theorem \ref{stl}, we may assume that \(\mathcal{A}\subseteq\binom{[n]}{k}\) and \(\mathcal{B}\subseteq\binom{[m]}{l}\) are two left-compressed cross-\(t\)-intersecting families, and we may also assume that the product \(|\mathcal{A}||\mathcal{B}|\) is maximum. Let \( g(\mathcal{A}) \in G_{*}(\mathcal{A}) \) and \( g(\mathcal{B}) \in G_{*}(\mathcal{B}) \) such that \( s := \max\{s^{+}(g(\mathcal{A})), s^{+}(g(\mathcal{B}))\}\) is minimal. By Theorem \ref{AK97}, we have
\[
|\mathcal{A}||\mathcal{B}| \geq \binom{n - t}{k - t}\binom{m - t}{l - t} \geq |\mathcal{F}(n, k, t, 1)||\mathcal{F}(m, l, t, 1)|.
\]
If \( s = t \), the cross-\(t\)-intersection property of \( g(\mathcal{A}) \) and \( g(\mathcal{B}) \) implies \( g(\mathcal{A}) = g(\mathcal{B}) = \{[t]\} \). Hence
\[
\mathcal{A} = \left\{ A \in \binom{[n]}{k} : [t] \subset A \right\}, \mathcal{B} = \left\{ B \in \binom{[m]}{l} : [t] \subset B \right\} \text{ and } |\mathcal{A}||\mathcal{B}| = \binom{n - t}{k - t}\binom{m - t}{l - t}.
\]
Therefore, for the remainder of the proof, we assume \( s \geq t + 1 \). Let \( i(t\leq i\leq n) \) be the smallest integer such that \( g_{i}^{*}(\mathcal{A}) \neq \emptyset \). Then, by (ii) of Lemma \ref{lemkl}, we have \( g_{s+t-i}^{*}(\mathcal{B}) \neq \emptyset \). 

We divide the proof into four cases based on the relationship between \(i, t\) and \(s\). 

{\bf Case 1}: \( i = t \). Then \( [s] \in g_{s+t-i}^{*}(\mathcal{B}) = g_{s}^{*}(\mathcal{B}) \). Since \( g(\mathcal{B}) \) is minimal (in the sense of set-theoretical inclusion), we have \( g(\mathcal{B}) = \{[s]\} \). By maximality of \( |\mathcal{A}||\mathcal{B}| \), we have  \( g(\mathcal{A}) = \binom{[s]}{t} \). Thus, 
\[
|\mathcal{A}| = \sum_{t \leq w \leq s} \binom{s}{w} \binom{n - s}{k - w} \text{ and } |\mathcal{B}| = \binom{m - s}{l - s}.
\]
Set \(\mathcal{A}_1 = \left\{A \in \binom{[n]}{k} : |A \cap [s-1]| \geq t\right\}\) and \(\mathcal{B}_1 = \left\{B \in \binom{[m]}{l} : [s-1] \subset B\right\}\). Clearly, \(\mathcal{A}_1\) and \(\mathcal{B}_1\) are cross-\(t\)-intersecting with
\[
|\mathcal{A}_1| = \sum_{t \leq w \leq s-1} \binom{s-1}{w} \binom{n-s+1}{k-w} \text{ and } |\mathcal{B}_1| = \binom{m-s+1}{l-s+1}.
\]
Then, we have
{\small \begin{align*}
&|\mathcal{A}| - |\mathcal{A}_1| =\sum\limits_{t \leq w \leq s} \binom{s}{w} \binom{n - s}{k - w}-\sum\limits_{t \leq w \leq s-1} \binom{s-1}{w} \binom{n-s+1}{k-w}\\
&=\sum\limits_{t \leq w \leq s-1} \left(\binom{s}{w} \binom{n - s}{k - w}-\binom{s-1}{w} \binom{n-s+1}{k-w}\right)+\binom{n - s}{k - s}\\
&=\sum\limits_{t \leq w \leq s-1} \left(\left(\binom{s-1}{w}+\binom{s-1}{w-1} \right)\binom{n - s}{k - w}-\binom{s-1}{w}\left( \binom{n-s}{k-w}+\binom{n-s}{k-w-1}\right)\right)+\binom{n - s}{k - s}\\
&=\sum\limits_{t \leq w \leq s-1} \left(\binom{s-1}{w-1} \binom{n - s}{k - w}-\binom{s-1}{w} \binom{n-s}{k-w-1}\right)+\binom{n - s}{k - s}\\
&=\binom{s-1}{t-1} \binom{n-s}{k-t}
\end{align*}}
and \(|\mathcal{B}_1| - |\mathcal{B}| = \binom{m-s}{l-s+1}\). It follows from the assumption \(m \geq (t+1)(k-t+1)\) and \(k\geq l\) that
\begin{align*}
s(m - l) - t(m - s + 1) &= (s - t)m - s(l - t) - t \geq (s - t)(t+1)(k-t+1) - s(l - t) - t \\
& = t(s - t - 1)(k - t + 1)+s(k-l)+s-t \geq 0.
\end{align*}
Therefore,
\begin{align*}
|\mathcal{A}_1||\mathcal{B}_1| - |\mathcal{A}||\mathcal{B}| &=\left(|\mathcal{A}|-\binom{s-1}{t-1} \binom{n-s}{k-t} \right)|\mathcal{B}_1|-|\mathcal{A}|\left(|\mathcal{B}_1|-\binom{m-s}{l-s+1}\right)\\
&=\binom{m-s}{l-s+1}\left(\sum_{t \leq w \leq s}\binom{s}{w} \binom{n - s}{k - w}\right)-\binom{s - 1}{t - 1}\binom{n - s}{k - t}\binom{m - s+1}{l - s+1}\\
&>\binom{n - s}{k - t}\left(\binom{s}{t} \binom{m-s}{l-s+1} -\binom{s-1}{t-1} \binom{m-s+1}{l-s+1} \right)\\
&=\binom{n - s}{k - t}\binom{s-1}{t-1} \binom{m-s}{l-s+1}  \left(\frac{s(m-l)-t(m-s+1)}{t(m-l)}\right)\geq 0,
\end{align*}
contradicting the maximality of \(|\mathcal{A}||\mathcal{B}|\). 

{\bf Case 2}: \(i=s\). Since \( g(\mathcal{A}) \) is minimal (in the sense of set-theoretical inclusion), we have \( g(\mathcal{A}) = \{[s]\} \). By maximality of \( |\mathcal{A}||\mathcal{B}| \), we have  \( g(\mathcal{B}) = \binom{[s]}{t} \). Thus, 
\[
|\mathcal{A}| = \binom{n - s}{k - s} \text{ and } |\mathcal{B}| = \sum_{t \leq w \leq s} \binom{s}{w} \binom{m - s}{l - w}.
\]
Set \(\mathcal{A}_1 =\left\{A \in \binom{[n]}{k} : [s-1] \subset A\right\}\) and \(\mathcal{B}_1 = \left\{B \in \binom{[m]}{l} : |A \cap [s-1]| \geq t\right\} \). Clearly, \(\mathcal{A}_1\) and \(\mathcal{B}_1\) are cross-\(t\)-intersecting with
\[
|\mathcal{A}_1| = \binom{n-s+1}{k-s+1} \text{ and } |\mathcal{B}_1| = \sum_{t \leq w \leq s-1} \binom{s-1}{w} \binom{m-s+1}{l-w} .
\]
By a similar argument as in Case 1, we obtain 
\begin{align*}
|\mathcal{A}_1||\mathcal{B}_1| - |\mathcal{A}||\mathcal{B}|> 0,
\end{align*}
contradicting the maximality of \(|\mathcal{A}||\mathcal{B}|\). Therefore, for the remainder of the proof, we assume \( s-1\geq i \geq t + 1 \).


{\bf Case 3}: \( s = t + 2 \). In this case, we have \( i =t + 1 \). Then we have \( g_{t+1}^{*}(\mathcal{A}) \neq \emptyset \) and \( g_{t+1}^{*}(\mathcal{B}) \neq \emptyset \). If \([t] \in g(\mathcal{A}) \cap g(\mathcal{B})\), then \( g(\mathcal{A}) = g(\mathcal{B}) = \{[t]\}\), which implies \( s = s^{+}(g(\mathcal{A})) = s^{+}(g(\mathcal{B})) = t \), yielding a contradiction. Therefore, we have \([t] \not\in g(\mathcal{A})\) or \([t] \not\in g(\mathcal{B})\).  We divide the proof into three subcases based on the relationship between \([t]\) and \(g(\mathcal{A}), g(\mathcal{B})\). 

{\bf Subcase 1}:\ If \([t] \not\in g(\mathcal{A})\) and \([t]\in g(\mathcal{B})\),  by maximality of $|\mathcal{A}||\mathcal{B}|$, we have
\[
g(\mathcal{B}) = \{[t]\} \cup \{[t+2] \setminus \{i\} : 1 \leq i \leq t\} \text{ and } g(\mathcal{A}) = \{[t+1], [t+2] \setminus \{t+1\}\}.
\]
Then we have 
\[
\mathcal{B} = \left\{B \in \binom{[m]}{l} : [t] \subset B \text{ or } |B \cap [t+2]| \geq t+1\right\} \text{ and}
\]
\[
\mathcal{A} = \left\{A \in \binom{[n]}{k} : [t] \subset A \text{ and } |A \cap [t+2]| \geq t+1\right\}.
\]
It follows that 
\[|\mathcal{A}||\mathcal{B}| = \left( \binom{n-t}{k-t} - \binom{n-t-2}{k-t} \right)\left( \binom{m-t}{l-t} + t \binom{m-t-2}{l-t-1} \right) .\]
We have 
\begin{align*}
\vspace{0.1cm}&t\binom{m-t-2}{l-t-1}\binom{n-t}{k-t}-\binom{m-t}{l-t}\binom{n-t-2}{k-t}-t\binom{m-t-2}{l-t-1}\binom{n-t-2}{k-t}\\
\vspace{0.1cm}&=t\binom{m-t-2}{l-t-1}\binom{n-t}{k-t}
\left(1-\frac{(m-t)(m-t-1)(n-k)(n-k-1)}{t(l-t)(m-l)(n-t)(n-t-1)}-\frac{(n-k)(n-k-1)}{(n-t)(n-t-1)} \right)\\
\vspace{0.1cm}&\leq t\binom{m-t-2}{l-t-1}\binom{n-t}{k-t}
\left(1-\frac{(m-t)(m-t-1)(tk-t^2+1)}{(t+1)(l-t)(m-l)((t+1)(k-t)+1)}-\frac{t(tk-t^2+1)}{(t+1)((t+1)(k-t)+1)} \right)\\
\vspace{0.1cm}&\leq t\binom{m-t-2}{l-t-1}\binom{n-t}{k-t}
\left(1-\frac{(m-t)(m-t-1)t}{(t+1)(l-t)(m-l)(t+1)}-\frac{t^2}{(t+1)^2} \right)\\
\vspace{0.1cm}&= t\binom{m-t-2}{l-t-1}\binom{n-t}{k-t}
\left(\frac{(2t+1)(l-t)(m-l)-(m-t)(m-t-1)t}{(t+1)^2(l-t)(m-l)} \right),
\end{align*}
where the second inequality follows from 
\[ \frac{n-k}{n-t}\geq \frac{(t+1)(k-t+1)-k}{(t+1)(k-t+1)-t}=\frac{tk-t^2+1}{(t+1)(k-t)+1}\]
and
\[ \frac{n-k-1}{n-t-1}\geq  \frac{(t+1)(k-t+1)-k-1}{(t+1)(k-t+1)-t-1}=\frac{t}{t+1}, \]
the third inequality follows from \(\frac{tk-t^2+1}{(t+1)(k-t)+1}> \frac{t}{t+1}.\) 

Set \( \psi_1(m, l, t)=(2t+1)(l-t)(m-l)-t(m-t)(m-t-1). \) Since \( m \geq (t+1)(k-t+1) \) and \(k\geq l\), we have 
\begin{align*}
\frac{\partial \psi_1}{\partial m} &=(2t+1)(l-t)-t(2m-2t-1)\\
&\leq (2t+1)(l-t)-t(2(t+1)(k-t+1)-2t-1)\\
&= -(2t+1)(k-l)-t-(k-t)(2t^2-1)\\
&<0,
\end{align*}
so \( \psi_1(m, l, t)\) is a decreasing function in \( m\). Then we have 
\begin{align*}
\psi_1(m, l, t) & \leq  \psi_1((t+1)(k-t+1), l, t)\\
&=\left((2t+1)(l-t)((t+1)(k-t+1)-l)-((t+1)(k-t)+1)(k-t)t(t+1)\right)\\
&=(l-t)((t+1)(k-t+1)-l)\left(2t+1-t(t+1)\cdot \frac{k-t}{l-t}\cdot \frac{(t+1)(k-t)+1}{(t+1)(k-t+1)-l}\right)\\
&\leq (l-t)((t+1)(k-t+1)-l)\left(2t+1-t(t+1)\right)\\
&< 0,\  
\end{align*}
where the penultimate inequality follows from \(k>l\) and \((t+1)(k-t)+1> (t+1)(k-t+1)-l.\)
This implies that
\begin{align*}
t\binom{m-t-2}{l-t-1}\binom{n-t}{k-t}<\binom{m-t}{l-t}\binom{n-t-2}{k-t}+t\binom{m-t-2}{l-t-1}\binom{n-t-2}{k-t}. 
\end{align*}
It follows that 
\begin{align*}
|\mathcal{A}||\mathcal{B}| = \left( \binom{n-t}{k-t} - \binom{n-t-2}{k-t} \right)\left( \binom{m-t}{l-t} + t \binom{m-t-2}{l-t-1} \right)< \binom{n-t}{k-t}\binom{m-t}{l-t},
\end{align*}
yielding a contradiction.

{\bf Subcase 2}:\ If \([t] \in g(\mathcal{A})\) and \([t] \not\in g(\mathcal{B})\),  by maximality of $|\mathcal{A}||\mathcal{B}|$, we have
\[
g(\mathcal{A}) = \{[t]\} \cup \{[t+2] \setminus \{i\} : 1 \leq i \leq t\} \text{ and } g(\mathcal{B}) = \{[t+1], [t+2] \setminus \{t+1\}\}.
\]
It follows that 
\begin{align*}
|\mathcal{A}||\mathcal{B}| = \left( \binom{n-t}{k-t}+t \binom{n-t-2}{k-t-1} \right)\left( \binom{m-t}{l-t}- \binom{m-t-2}{l-t} \right).
\end{align*}
We have 
\begin{align*}
\vspace{0.1cm}&t\binom{n-t-2}{k-t-1}\binom{m-t}{l-t}-\binom{n-t}{k-t}\binom{m-t-2}{l-t}-t\binom{n-t-2}{k-t-1}\binom{m-t-2}{l-t}\\
\vspace{0.1cm}&=t\binom{n-t-2}{k-t-1}\binom{m-t}{l-t}
\left(1-\frac{(n-t)(n-t-1)(m-l)(m-l-1)}{t(k-t)(n-k)(m-t)(m-t-1)}-\frac{(m-l)(m-l-1)}{(m-t)(m-t-1)} \right)\\
\vspace{0.1cm}&\leq t\binom{n-t-2}{k-t-1}\binom{m-t}{l-t}
\left(1-\frac{(n-t)(n-t-1)(tk-t^2+1)}{(t+1)(k-t)(n-k)((t+1)(k-t)+1)}-\frac{t(tk-t^2+1)}{(t+1)((t+1)(k-t)+1)} \right)\\
\vspace{0.1cm}&\leq t\binom{n-t-2}{k-t-1}\binom{m-t}{l-t}
\left(1-\frac{(n-t)(n-t-1)t}{(t+1)(k-t)(n-k)(t+1)}-\frac{t^2}{(t+1)^2} \right)\\
\vspace{0.1cm}&= t\binom{n-t-2}{k-t-1}\binom{m-t}{l-t}
\left(\frac{(2t+1)(k-t)(n-k)-(n-t)(n-t-1)t}{(t+1)^2(k-t)(n-k)} \right),
\end{align*}
where the second inequality follows from 
\[ \frac{m-l}{m-t}\geq \frac{(t+1)(k-t+1)-k}{(t+1)(k-t+1)-t}=\frac{tk-t^2+1}{(t+1)(k-t)+1}\]
and
\[ \frac{m-l-1}{m-t-1}\geq  \frac{(t+1)(k-t+1)-k-1}{(t+1)(k-t+1)-t-1}=\frac{t}{t+1}, \]
the third inequality follows from \(\frac{tk-t^2+1}{(t+1)(k-t)+1}> \frac{t}{t+1}.\) 

Set \( \psi_2(n, k, t)=(2t+1)(k-t)(n-k)-t(n-t)(n-t-1). \) Since \( n \geq (t+1)(k-t+1) \), we have 
\begin{align*}
\frac{\partial \psi_2}{\partial n} &=(2t+1)(k-t)-t(2n-2t-1)\\
&\leq (2t+1)(k-t)-t(2(t+1)(k-t+1)-2t-1)\\
&= -t-(k-t)(2t^2-1)\\
&<0,
\end{align*}
so \( \psi_2(n, k, t)\) is a decreasing function in \(n\). Then we have 
\begin{align*}
\psi_2(n, k, t) & \leq  \psi_1((t+1)(k-t+1), k, t)\\
&=\left((2t+1)(k-t)((t+1)(k-t+1)-k)-((t+1)(k-t)+1)(k-t)t(t+1)\right)\\
&=(k-t)(tk-t^2+1)\left(2t+1-t(t+1)\cdot \frac{(t+1)(k-t)+1}{t(k-t)+1}\right)\\
&\leq (k-t)(tk-t^2+1)\left(2t+1-t(t+1)\right)\\
&< 0. 
\end{align*}
This implies that
\begin{align*}
t\binom{n-t-2}{k-t-1}\binom{m-t}{l-t}< \binom{n-t}{k-t}\binom{m-t-2}{l-t}+t\binom{n-t-2}{k-t-1}\binom{m-t-2}{l-t}.
\end{align*}
It follows that 
\begin{align*}
|\mathcal{A}||\mathcal{B}| = \left( \binom{n-t}{k-t}+t \binom{n-t-2}{k-t-1} \right)\left( \binom{m-t}{l-t}- \binom{m-t-2}{l-t} \right)< \binom{n-t}{k-t}\binom{m-t}{l-t}, 
\end{align*}
yielding a contradiction.

{\bf Subcase 3}:\ If \([t] \not\in g(\mathcal{A})\) and \([t] \not\in g(\mathcal{B})\), by (i) of Lemma \ref{lemkl}, we have \(|F| \geq t+1\) for all \(F \in g(\mathcal{A}) \cup g(\mathcal{B})\). It follows from the maximality of $|\mathcal{A}||\mathcal{B}|$ that 
\[
g(\mathcal{A}) = \binom{[t+2]}{t+1}\ \text{and}\ g(\mathcal{B}) = \binom{[t+2]}{t+1}.
\]
Hence,
\[
\mathcal{A} = \mathcal{F}(n, k, t, 1) = \left\{A \in \binom{[n]}{k} : A \cap [t+2]| \geq t+1\right\} \text{ and} 
\] 
\[
\mathcal{B} = \mathcal{F}(n, l, t, 1) = \left\{B \in \binom{[m]}{l} : B \cap [t+2]| \geq t+1\right\}.
\]  
If \(\min\{m, n\} > (t+1)(k-t+1)\), by Theorem \ref{AK97}, we have \(\binom{n - t}{k - t}\binom{m - t}{l - t}> |\mathcal{F}(n, k, t, 1)||\mathcal{F}(m, l, t, 1)|,\) yielding a contradiction. 
Hence, the maximality of \(|\mathcal{A}||\mathcal{B}|\) implies that \(m=n = (t+1)(k-t+1)\) and \( k=l\). 


{\bf Case 4}: \( s \geq t + 3 \). Then, by (ii) of Lemma \ref{lemkl}, we have \( g_{s+t-i}^{*}(\mathcal{B}) \neq \emptyset \). Set \(\mathcal{A}_1 = \mathcal{A} \cup \mathcal{D}(g_{i}^{*}(\mathcal{A})')\) and \(\mathcal{B}_1 = \mathcal{B} \setminus \mathcal{D}(g_{s+t-i}^{*}(\mathcal{B}))\). According to (iii) of Lemma \ref{lemkl}, \(\mathcal{A}_1\) and \(\mathcal{B}_1\) are cross-\(t\)-intersecting, and
\[
|\mathcal{A}_1||\mathcal{B}_1| = \left( |\mathcal{A}| + |g_{i}^{*}(\mathcal{A})| \binom{n-s}{k-i+1} \right) \left( |\mathcal{B}| - |g_{s+t-i}^{*}(\mathcal{B})| \binom{m-s}{l-s-t+i} \right) \leq |\mathcal{A}||\mathcal{B}|,
\]
where the inequality holds by the maximality of \(|\mathcal{A}||\mathcal{B}|\). This implies that 
\begin{equation}\label{gs11}
\frac{|\mathcal{B}|}{|\mathcal{A}| + |g_{i}^{*}(\mathcal{A})| \binom{n-s}{k-i+1}}  \leq \frac{|g_{s+t-i}^{*}(\mathcal{B})| \binom{m-s}{l-s-t+i}}{|g_{i}^{*}(\mathcal{A})| \binom{n-s}{k-i+1}}
\end{equation}
and
\begin{equation}\label{gs12}
\frac{|\mathcal{B}|}{|g_{s+t-i}^{*}(\mathcal{B})|} \binom{n-s}{k-i+1} \leq \frac{|\mathcal{A}| }{|g_{i}^{*}(\mathcal{A})|}\binom{m-s}{l-s-t+i}+\binom{n-s}{k-i+1}\binom{m-s}{l-s-t+i}.
\end{equation}
Set \(\mathcal{A}_2 = \mathcal{A} \setminus \mathcal{D}(g_{i}^{*}(\mathcal{A}))\) and \(\mathcal{B}_2 = \mathcal{B} \cup \mathcal{D}(g_{s+t-i}^{*}(\mathcal{B})')\). Similarly, \(\mathcal{A}_2\) and \(\mathcal{B}_2\) are also cross-\(t\)-intersecting families with
\[
|\mathcal{A}_2||\mathcal{B}_2| = \left( |\mathcal{A}| - |g_{i}^{*}(\mathcal{A})| \binom{n-s}{k-i} \right) \left( |\mathcal{B}| + |g_{s+t-i}^{*}(\mathcal{B})| \binom{m-s}{l-s-t+i+1} \right) \leq |\mathcal{A}||\mathcal{B}|,
\]
and then it holds that 
\begin{equation}\label{gs21}
\frac{|\mathcal{A}|}{|\mathcal{B}| + |g_{s+t-i}^{*}(\mathcal{B})| \binom{m-s}{l-s-t+i+1}}  \leq \frac{|g_{i}^{*}(\mathcal{A})| \binom{n-s}{k-i}}{|g_{s+t-i}^{*}(\mathcal{B})| \binom{m-s}{l-s-t+i+1}}
\end{equation}
and
\begin{equation}\label{gs22}
\frac{|\mathcal{A}| }{|g_{i}^{*}(\mathcal{A})|}\binom{m-s}{l-s-t+i+1}\leq \frac{|\mathcal{B}|}{|g_{s+t-i}^{*}(\mathcal{B})|} \binom{n-s}{k-i} +\binom{n-s}{k-i}\binom{m-s}{l-s-t+i+1}.
\end{equation}

Combining  (\ref{gs11}) and (\ref{gs21}), we have
\begin{equation}\label{gs3}
\frac{|\mathcal{A}|}{|\mathcal{A}| + |g_{i}^{*}(\mathcal{A})| \binom{n-s}{k-i+1}} \cdot \frac{|\mathcal{B}|}{|\mathcal{B}| + |g_{s+t-i}^{*}(\mathcal{B})| \binom{m-s}{l-s-t+i+1}}  \leq \frac{\binom{m-s}{l-s-t+i} \binom{n-s}{k-i}}{\binom{n-s}{k-i+1} \binom{m-s}{l-s-t+i+1}}.
\end{equation}

Set \(\nabla(g_{i}^{*}(\mathcal{A})')=\left\{F \in\binom{[s - 1]}{i}: E \subset F\text{ for some }E \in g_{i}^{*}(\mathcal{A})' \right\}\). Since \(\mathcal{A}\) is left-compressed, it follows that \( \mathcal{D}(\nabla(g_{i}^{*}(\mathcal{A})')) \subseteq \mathcal{A}\). By Theorem \ref{NMP}, we have 
\[\frac{|\nabla(g_{i}^{*}(\mathcal{\mathcal{A}})')|}{|g_{i}^{*}(\mathcal{A})|}=\frac{|\nabla(g_{i}^{*}(\mathcal{A})')|}{|g_{i}^{*}(\mathcal{A})'|} \geq \frac{\binom{s - 1}{i}}{\binom{s - 1}{i - 1}},\]
and then we obtain
\begin{equation}\label{gs6}
|\mathcal{A}| \geq \mathcal{D}(\nabla(g_{i}^{*}(\mathcal{A})') \cup g_{i}^{*}(\mathcal{A})) \geq \frac{|g_{i}^{*}(\mathcal{A})|}{\binom{s - 1}{i - 1}}\left(\binom{s - 1}{i - 1}\binom{n - s}{k - i}+\binom{s - 1}{i}\binom{n - s + 1}{k - i}\right).
\end{equation}
Hence, 

\begin{equation}\label{gs7}
\frac{|\mathcal{A}|}{|\mathcal{A}| + |g_{i}^{*}(\mathcal{A})| \binom{n-s}{k-i+1}}\geq
\frac{\binom{s - 1}{i - 1}\binom{n - s}{k - i}+\binom{s - 1}{i}\binom{n - s + 1}{k - i}}{\binom{s - 1}{i - 1}\binom{n - s+1}{k - i+1}+\binom{s - 1}{i}\binom{n - s + 1}{k - i}}=\frac{(k-i+1)S(n,s,i,k)}{(n-s+1)T(n,s,i,k)}.
\end{equation}
Similarly, we have
\begin{equation}\label{gs8} 
|\mathcal{B}| \geq \frac{|g_{s+t-i}^{*}(\mathcal{B})|}{\binom{s - 1}{s+t-i - 1}}\left(\binom{s - 1}{s+t-i - 1}\binom{m - s}{l-s-t +i}+\binom{s - 1}{s+t-i}\binom{m - s + 1}{l- s-t +i}\right)
\end{equation}
and then it holds that
\begin{equation}\label{gs9} 
\begin{aligned}
\vspace{2cm}\frac{|\mathcal{B}|}{|\mathcal{B}| + |g_{s+t-i}^{*}(\mathcal{B})| \binom{m-s}{l-s-t+i+1}}  & \geq \frac{\binom{s - 1}{s+t-i - 1}\binom{m - s}{l-s-t +i}+\binom{s - 1}{s+t-i}\binom{m - s + 1}{l- s-t +i}}{\binom{s - 1}{s+t-i - 1}\binom{m - s+1}{l-s-t +i+1}+\binom{s - 1}{s+t-i}\binom{m - s + 1}{l- s-t +i}}\\
&=\frac{(l-s-t+i+1)S(m,s,s+t-i,l)}{(m-s+1)T(m,s,s+t-i,l)}.
\end{aligned}
\end{equation}
Combining (\ref{gs3}), (\ref{gs7}) and (\ref{gs9}), we have 
\begin{align*}
\frac{\binom{m-s}{l-s-t+i} \binom{n-s}{k-i}}{\binom{n-s}{k-i+1} \binom{m-s}{l-s-t+i+1}} &\geq \frac{|\mathcal{A}|}{|\mathcal{A}| + |g_{i}^{*}(\mathcal{A})| \binom{n-s}{k-i+1}} \frac{|\mathcal{B}|}{|\mathcal{B}| + |g_{s+t-i}^{*}(\mathcal{B})| \binom{m-s}{l-s-t+i+1}}\\
&\geq \frac{(l-s-t+i+1)(k-i+1)S(n,s,i,k)S(m,s,s+t-i,l)}{(m-s+1)(n-s+1)T(n,s,i,k)T(m,s,s+t-i,l)}.
\end{align*}

We divide the proof into three subcases based on the relationship between \(s, t\) and \(i\).

{\bf Subcase 1}: \(s-2\geq i\geq t+2.\) By Lemma \ref{IM1}, we have
\begin{align*}
\frac{\binom{m-s}{l-s-t+i} \binom{n-s}{k-i}}{\binom{n-s}{k-i+1} \binom{m-s}{l-s-t+i+1}} < \frac{(l-s-t+i+1)(k-i+1)S(n,s,i,k)S(m,s,s+t-i,l)}{(m-s+1)(n-s+1)T(n,s,i,k)T(m,s,s+t-i,l)},
\end{align*}
yielding a contradiction.  

{\bf Subcase 2}: \(i=s-1.\) Since \(k\geq l, k\geq i\geq t+1\) and \(\min\{m, n\}\geq (t+1)(k-t+1) \), we have 
\[
(m-l+t-i)-(k-i+1)=m-l-k+t-1\geq (t-1)(k-t)+(k-l)> 0,\ {\rm and} 
\]
\[
(n-s-k+i)-(l-s-t+i+1)=n-l-k+t-1\geq(k-l)+(t-1)(k-t)>0. 
\]
This implies that 
\[(m-l+t-i)(n-s-k+i)-(k-i+1)(l-s-t+i+1)>0.\]
Multiplying (\ref{gs22}) by \(\frac{l-s-t+i+1}{m-l+t-i}\) and adding it to (\ref{gs12}) yields
\begin{equation}\label{gs05}
\frac{|\mathcal{B}|}{|g_{s+t-i}^{*}(\mathcal{B})|}\leq \binom{m-s}{l-t-s+i}\frac{(n-s+1)(m-l+t-i)}{(m-l+t-i)(n-s-k+i)-(k-i+1)(l-s-t+i+1)}.
\end{equation}
Combining  (\ref{gs05}) and (\ref{gs8}), we have 
\begin{align*}
&\ \ 1+\frac{(i-t)(m-s+1)}{(s+t-i)(m-l+t-i+1)}\\
&\leq \frac{|\mathcal{B}|}{|g_{s+t-i}^{*}(\mathcal{B})|}\cdot \binom{m-s}{l-t-s+i}^{-1}\\
&\leq \frac{(n-s+1)(m-l+t-i)}{(m-l+t-i)(n-s-k+i)-(k-i+1)(l-s-t+i+1)},
\end{align*}
which is equivalent to
\begin{equation}\label{gs91} 
(i-t)\big((m-l+t-i)(n-s-k+i)-(k-i+1)(l-s-t+i+1)\big)\leq (s+t-i)(k-i+1)(m-l+t-i+1).
\end{equation}
Substitute \(i=s-1\) into the above equation, we have 
\[(t+1)(k-s+2)(m-l+t-s+2)-(s-1-t)\big((m-l+t-s+1)(n-k-1)-(k-s+2)(l-t)\big)\geq 0,\]
a contradiction with Lemma \ref{IM2}. 

{\bf Subcase 3}: \(i=t+1.\) Similar to (\ref{gs05}), multiplying (\ref{gs12}) by \(\frac{k-i+1}{n-s-k+i}\) and adding it to (\ref{gs22}) yields
\begin{equation}\label{gs040}
\frac{|\mathcal{A}| }{|g_{i}^{*}(\mathcal{A})|}\leq \binom{n-s}{k-i}\frac{(m-s+1)(n-s-k+i)}{(m-l+t-i)(n-s-k+i)-(k-i+1)(l-s-t+i+1)}.
\end{equation}
Combining (\ref{gs040}) and (\ref{gs6}), we have
\begin{align*}
&1+\frac{(s-i)(n-s+1)}{i(n-s-k+i+1)}\\
&\leq \frac{|\mathcal{A}| }{|g_{i}^{*}(\mathcal{A})|}\cdot \binom{n-s}{k-i}^{-1}\\
&\leq \frac{(m-s+1)(n-s-k+i)}{(m-l+t-i)(n-s-k+i)-(k-i+1)(l-s-t+i+1)},
\end{align*}
which is equivalent to
\begin{equation}\label{gs81} 
(s-i)\big((m-l+t-i)(n-s-k+i)-(k-i+1)(l-s-t+i+1)\big)\leq i(n-s-k+i+1)(l-s-t+i+1).
\end{equation}
Substitute \(i=t+1\) into the above equation, we have 
\[(t+1)(n-s-k+t+2)(l-s+2)-(s-1-t)\big((m-l-1)(n-s-k+t+1)-(k-t)(l-s+2)\big)\geq 0,\]
a contradiction with Lemma \ref{IM3}. This completes the proof. \qed

\section{Acknowledgment}

The first author was supported by National Natural Science Foundation of China Grant no. 12471313. The second author was supported by National Natural Science Foundation of China Grant no. 12271390.


\begin{thebibliography}{Z} \baselineskip 11pt

\bibitem{AK1996}
R. Ahlswede, L.H. Khachatrian, The complete nontrivial-intersection theorem for systems of finite sets, \emph{J. Comb. Theory, Ser. A} \textbf{76} (1996) 121-138.

\bibitem{AK1997}
R. Ahlswede, L.H. Khachatrian, The complete intersection theorem for systems of finite sets, \emph{Eur. J. Comb.} \textbf{18}(2) (1997) 125-136.

\bibitem{B2014}
P. Borg, The maximum product of sizes of cross-\(t\)-intersecting uniform families, \emph{Australas. J. Comb.} \textbf{60}(1) (2014) 69-78.

\bibitem{B2016}
P. Borg, The maximum product of weights of cross-intersecting families, \emph{J. Lond. Math. Soc.} \textbf{94} (2016) 993-1018.

\bibitem{CLWZ2025}
Y. Chen, A. Li, B. Wu and H. Zhang, On cross-$2$-intersecting families, preprint (2025). arXiv:2503.15971.

\bibitem{EKR1961}
P. Erd\H{o}s, C. Ko, R. Rado, Intersection theorems for systems of finite sets, \emph{Q. J. Math.} \textbf{2} (1961) 313-320.

\bibitem{F1978}
P. Frankl, The Erd\H{o}s-Ko-Rado theorem is true for $n = ckt$, In: \emph{Combinatorics, Proc. Fifth Hungarian Colloq., vol. I}, Keszthely, 1976, Colloq. Math. Soc. János Bolyai, vol. 18, North-Holland, Amsterdam-New York, 1978, pp. 365-375.

\bibitem{F1987}
P. Frankl, The shifting technique in extremal set theory, In: C. Whitehead (Ed.), \emph{Surveys in Combinatorics}, LMS Lecture Note Series, vol. 123, Cambridge Univ. Press, 1987, pp. 81-110.

\bibitem{FF1986}
P. Frankl, Z. Füredi, Non-trivial intersecting families, \emph{J. Comb. Theory, Ser. A} \textbf{41} (1986) 150-153.

\bibitem{FLST2014}
P. Frankl, S.J. Lee, M. Siggers, N. Tokushige, An Erd\H{o}s-Ko-Rado theorem for cross $t$-intersecting families,
 \emph{J. Comb. Theory, Ser. A} \textbf{128} (2014) 207-249.
 

\bibitem{G2010}
M. Gromov, Singularities, expanders and topology of maps. Part 2: from combinatorics to topology via algebraic isoperimetry, \emph{Geom. Funct. Anal.} \textbf{20} (2010) 416-526.

\bibitem{HLWZ2026}
D. He, A. Li, B. Wu, H. Zhang, On nontrivial cross-\(t\)-intersecting families, \emph{J. Comb. Theory, Ser. A} \textbf{217} (2026) 106095.

\bibitem{MT1989}
M. Matsumoto, N. Tokushige, The exact bound in the Erd\H{o}s-Ko-Rado theorem for cross-intersecting families, \emph{J. Comb. Theory, Ser. A} \textbf{52} (1989) 90-97.

\bibitem{P1986}
L. Pyber, A new generalization of the Erd\H{o}s-Ko-Rado theorem, \emph{J. Comb. Theory, Ser. A} \textbf{43} (1986) 85-90.

\bibitem{S1928}
E. Sperner, Ein satz über untermengen einer endlichen Menge, \emph{Math. Z.} \textbf{27} (1928) 544-548.

\bibitem{TT2025}
H. Tanaka and N. Tokushige, A semidefinite programming approach to cross-2-intersecting families, preprint (2025). arXiv:2503.14844.

\bibitem{T2010}
N. Tokushige, On cross \(t\)-intersecting families of sets, \emph{J. Combin. Theory Ser. A} \textbf{117} (2010) 1167-1177.

\bibitem{T2013}
N. Tokushige, The eigenvalue method for cross-\(t\)-intersecting families, \emph{J. Algebraic Comb.} \textbf{38} (2013) 653-662.


\bibitem{W1984}
R.M. Wilson, The exact bound in the Erd\H{o}s-Ko-Rado theorem, \emph{Combinatorica} \textbf{4} (1984) 247-257.

\bibitem{ZW2025}
H. Zhang, B. Wu, On a conjecture of Tokushige for cross-\(t\)-intersecting families,  \emph{J. Comb. Theory, Ser. B} \textbf{171} (2025) 49-70.

\end{thebibliography}
\end{document}